\documentclass[notitlepage,reqno,11pt]{amsart}
\usepackage{amsmath}
\usepackage{amssymb}
\usepackage{amsfonts,amsthm,hyperref}
\usepackage[margin=1in]{geometry}
\usepackage{xcolor}
\usepackage[normalem]{ulem}
\usepackage[foot]{amsaddr}
\usepackage{mathtools}

\usepackage{accents}

\usepackage{soul}

\usepackage{enumitem}

\numberwithin{equation}{section}
\newtheorem{assumption}{Assumption}[section]
\newtheorem{lemma}{Lemma}[section]
\newtheorem{theorem}{Theorem}[section]
\newtheorem{definition}{Definition}[section]
\newtheorem{coro}{Corollary}[section]
\newtheorem{prop}{Proposition}[section]
\newtheorem{remark}{Remark}[section]

\newcommand{\RR}{{\mathbb R}}

\newcommand{\uW}{{\underline{\textit{\makebox[.7em]{W}}}}}

\def\E{\mathbb{E}}

\definecolor{ISUred}{RGB}{200, 16, 46}

\newcommand{\Poisson}{{\mathsf{P}}}

\newcommand{\ttl}{\Large Growing open Markovian Jackson networks: Fluid limit \\[5pt] and infinite-dimensional Skorokhod problem}

\newcommand{\ttls}{ \large Growing open Markovian Jackson networks}

\allowdisplaybreaks

\begin{document}

\title[\ttls]{\ttl}

\author{Louis T. Clarke$^*$, Guodong Pang$^*$ and Ruoyu Wu$^{**}$}

\address{$^*$Department of Computational Applied Mathematics and Operations Research,
						George R. Brown School of Engineering and Computing, 
						Rice University,
						Houston, TX 77005}
		\email{lc160,\,gdpang@rice.edu}
\address{$^{**}$Department of Mathematics,
						Iowa State University, Ames, IA 50011
						}
		\email{ruoyu@iastate.edu}
		
\date{\today}

\begin{abstract} 

We study growing open Jackson networks where each station is a single-server queue that follows the first-come first-served discipline with Poisson arrivals and exponentially distributed service times, characterized by node-specific rates.
In applying a fluid scaling to the queue-length process,  
we show that under certain conditions the queueing system can be approximated by an infinite-dimensional fluid limit with a kernel function in place of the transition matrix. This limiting process can be characterized by an infinite-dimensional Skorokhod problem, for which we develop a new theory by considering a broader class of reflection operators and general infinite-dimensional processes. 
We establish existence and uniqueness of a solution along with Lipschitz continuity provided the reflecting operator has a spectral radius less than 1.
By introducing an intermediate process in which the compensated Poisson components are removed, and then lifting this to an infinite-dimensional process, 
we exploit the new Lipschitz property of the infinite-dimensional Skorokhod mapping to prove convergence of the intermediate process. We then prove the necessary estimates for the difference between the original and intermediate processes by using martingale properties. Finally, we consider the empirical measure of the queueing processes, for which we show convergence to the measure associated with the path of the infinite-dimensional fluid limit, extending to the convergence of specific performance-related functionals.
\end{abstract}

\keywords{Growing open Jackson network, fluid limit, empirical measure, infinite-dimensional Skorokhod problem, reflection operator, existence and uniqueness, Lipschitz continuity,  space $D([0,T],L_1)$}

\maketitle

\section{Introduction}
\label{sec:intro}  
Since its introduction \cite{Jackson1957}, Jackson networks have been widely used in various applications, such as manufacturing and production, telecommunication, healthcare and service systems. We refer the readers to the manuscripts on queueing networks such as \cite{ChenYao2001fundamentals, Whitt2002,serfozo2012, robert2013, kelly2011reversibility, kelly2014stochastic,srikant2014communication} for comprehensive overviews. In standard Jackson networks, it is assumed that there is a fixed number of service stations (or nodes), along with a fixed network topology (connectivity between the nodes), and the stochastic flow dynamics of jobs in these networks is the focus of their study.

In open Jackson networks of single-server queues, there are external arrival processes at each service station, and jobs are served in the first-come first-served discipline, and after service completion, jobs may leave the system or be routed randomly to other service stations (possibly back to its own station). 
In the Markovian setting, assuming that the arrival processes are Poisson and the service times are exponentially distributed, a well-known result is the product-form stationary distribution for the joint queue length processes (see, e.g., \cite{Jackson1957,ChenYao2001fundamentals}). 
In heavy traffic, it is established that the queue-length processes, under the law of large numbers (LLN) scaling, converge in probability to the fluid queueing limit, which is characterized by the solution to the finite-dimensional Skorokhod problem (see, e.g., \cite{ChenYao2001fundamentals, Whitt2002}).  
There are various extensions of Jackson networks, including non-Markovian generalized Jackon networks with renewal arrivals and independent and identically distributed (i.i.d.) general service times at each station (see, e.g., \cite{GZ2006}), and Jackson networks with abandonment/reneging (see, e.g., \cite{huang2013diffusion}).

There also exist infinite Markovian open Jackson networks, where it is assumed that the number of service stations is infinite; see, e.g., \cite{kel1989jackson,kelbert1990central}, where the authors study the existence of an invariant measure and central limit theorem for the queue length processes. 
In \cite{stolyar1989asymptotic,fayolle1996,borovkov1998propagation,khmelev2001infinite}, the authors study the stability and mean-field approximations of closed Jackson networks where the numbers of customers and servers grow together to infinity such that their ratio converges to a constant. 
In \cite{rybko2005poisson1, rybko2005poisson2}, the authors study a growing closed and complete Jackson network with uniform transition probabilities to all other nodes, and consider the stationary distribution at any given node. Known as the Poisson hypothesis, they show that certain flow rates can be approximated by Poisson flow rates. This is expanded upon in \cite{rybko2013stationary}, where instead the Jackson network is open and more general transition rates are permitted. Here, stations are split into a fixed number of types, and routing to stations of the same type is uniform. In this way, as the network grows larger, there is not a transition kernel increasing in dimension to infinity as we study in our work.

In this paper, we study open Jackson networks of single-server queues, with independent external Poisson arrivals and i.i.d.\ exponential service times at each station, and with a Markovian routing scheme after service completion at each station. 
Specifically, we start with a finite connected network on $N$ stations, where $\lambda_i^N,\mu_i^N$ are the arrival and service rates at node $i$ respectively, and  $P^N$ is the routing matrix. The object of interest is the $N$-dimensional queue length process $Q^N=\{(Q_1^N(t) \dots, Q_N^N(t)): t\ge0\}$ where $Q_i^N(t)$ represents the number of jobs waiting at station $i$ at time $t$. 
We consider the fluid-scaled queueing process  $\bar{Q}^N(t) = N^{-\alpha} Q^N(N^\alpha t)$ for any $\alpha>0$,  together with an appropriate scaling on the routing matrix $P^N$. We assume that the underlying graph for the connectivity of the nodes from the routing matrix $P^N$ is dense, so that $P^N$ is scaled as of $O(1/N)$ and there exists a nonnegative function $G(u,v)$ over $[0,1]^2$ as the limit of $N P^N$ in the operator norm (see Assumption \ref{assumption:GNGop}; here $G(u,v)$ resembles a graphon but can take any nonnegative values). 
Observe that the dimension of the vector $Q^N$ grows as $N$ increases to infinity.  Hence, unlike the standard fluid analysis, we expect an infinite-dimensional fluid limit as an approximation in a proper sense. 

In the standard study with a fixed number of stations, the fluid limit is given by a finite-dimensional Skorokhod problem, with $(I-P^T)$ as the reflection matrix, where $I$ is the identity matrix and $P^T$ is the transpose of the routing matrix.  
For the growing network model, the fluid limit, denoted as  $\bar{Q}=\{\bar{Q}_u(t): t\ge 0, u \in [0,1]\}$, is characterized by an infinite-dimensional Skorokhod problem, with the corresponding  $({\bf1}-\boldsymbol{G}^T)$ as the reflection operator, where $\boldsymbol{G}$ is the integral operator associated with the function $G(u,v)$. We regard the limit $\bar{Q}_u(t)$ in the space $\mathcal{D}_T(L_1) \coloneq D([0,T], L_1[0,1])$.  
To understand the existence and uniqueness of a solution to the fluid limit $\bar{Q}_u(t)$, we develop a new theoretical framework for infinite-dimensional Skorokhod problems in $\mathcal{D}_T(L_1)$, applying the general theory to establish the well-definedness of the fluid limit $\bar{Q}_u(t)$.

Recall that for the finite-dimensional Skorokhod problems,  there exist two approaches to establish the existence and uniqueness when the network is open and the routing matrix has a spectral radius strictly less than 1. In \cite{Whitt2002}, the problem is first described as defining and finding the infimum of the regulator set. From this, existence, uniqueness and complementarity conditions can be established as well as continuity and lipschitz properties. In \cite{HarrisonReiman1981reflected}, the solution to the Skorokhod problem is achieved by using a contraction map made possible through a diagonalization argument on the routing matrix. 
In this paper, for the infinite-dimensional Skorokhod problem we establish the existence and uniqueness by adapting the first of these two methods, under the assumption that the spectral radius of the reflection operator is less than $1$. 
This requires establishing the regulator map as the pointwise infimum of the regulator set, which can be equivalently shown to be the unique fixed point of a mapping on $\mathcal{D}_T^\uparrow(L_1)$ (time-increasing elements of $\mathcal{D}_T(L_1)$) given in Definition \ref{def:pidefD} (see Theorem \ref{thm:unqfixpoint}). We also give the analogous complementarity characterization property for the reflection map (see Theorem \ref{thm:compchar}).
Moreover, we establish the Lipschitz continuity properties by adapting the methodology of \cite{Whitt2002} to the infinite-dimensional setting. 
This requires first showing that the $k$-fold composition of the reflection operator is bounded in operator norm for some finite $k$, which leads to results in Lipschitz continuity of both the free process and the reflection operator (see Theorem \ref{thm:lipschitz} and Lemma \ref{lem:krefbound}). Combining these two results in general Lipschitz continuity properties in Theorem \ref{thm:kconvergence}.
To facilitate the proof for the convergence of fluid-scaled queueing processes, we also prove similar bounds on convergent sequences of reflection operators in Lemma \ref{lem:opseqnorm}, which leads to Lipschitz continuity of converging processes (see  Theorem \ref{thm:kconvergenceseq}).

In order to prove convergence of the fluid-scaled queueing process $\bar{Q}^N(t)$, we introduce an intermediate fluid-scaled queueing process by removing the compensated Poisson components from the finite system, which then becomes deterministic. We establish the error bounds between the original and intermediate finite-dimensional fluid-scaled processes, by applying martingale inequalities (see Lemma \ref{lem:compbound}). Next, we construct the corresponding infinite-dimensional representation of the intermediate processes, as an infinite-dimensional Skorokhod problem, and employ the Lipschitz continuity property to conclude the convergence of the  infinite-dimensional intermediate fluid-scaled processes to the fluid limit (see Lemma \ref{lem:tQbQ-diff}). By combining these two convergence results we show that under sufficient assumptions on the convergence of $\lambda^N$, $\mu^N$ and $G^N$, the processes converge at a rate of $N^{-\alpha/2}$, where $\alpha$ is a positive constant chosen from the scaling assumptions (see Section \ref{thmprf:fullconv}). Moreover we also consider the empirical measure of the fluid-scaled queueing processes, which is shown to converge to a probability measure associated with the path of the infinite-dimensional fluid limit over the interval $[0,T]$. We discuss how this result applies to several functionals of the queueing process, see Remark \ref{rem:functionalmeasure}. The proof exploits the intermediate process and an argument relating the Wasserstein metric on probability measures to the metric for the queueing processes on $\mathcal{D}_T(L_1)$; see Section \ref{subsec:empconv}.

\subsection{Literature Review}
Our work is relevant to two recent streams of literature: graphon particle systems and queues on graphs. 
Our model can be also regarded as an interacting queueing system on graphs, since each node has its own dynamics as a single-server queue and there is an underlying network topology in the routing matrix. Hence, the interactions among the queues at each station are through the routing mechanism. 
In recent studies of interacting particle systems, the dynamics at each node is usually a (jump) diffusion, while the interaction is through a ``weighted" mean-field type functional in the drift and/or variance coefficient; see, e.g., 
\cite{BayraktarChakrabortyWu2023graphon,BetCoppiniNardi2020weakly,Lucon2020quenched,CoppiniDeCrescenzoPham2024nonlinear}. 
There is an underlying graph topology in the description of the weights; hence in the dense graph setting, under the proper scaling, the limiting processes can be described as a mean-field type of stochastic process on a graphon (capturing the interacting heterogeneity). However, in our model, under the scaling of the transition probability matrix, its limit is not given by a graphon, but rather a nonnegative kernel function over $[0,1]^2$ that takes values in $\RR_+$ and is not necessarily symmetric, which resembles the heterogeneity effect of a graphon. Hence, we do not work with cut-norm as in the literature of dense graph convergence to graphon but rather involving the $L^1$ convergence for the reflection operator. 
Moreover,  the processes of interest are very different, since the queueing processes are under the fluid scaling (with both time and space being scaled appropriately as described above) while there is not time scaling involved in the literature of interacting particle systems of (jump) diffusions. As a result, the proofs for the convergence of the processes of interest also differ significantly. In this paper convergence is proved by lifting the finite system to an infinite-dimensional representation via an intermediate compensated process instead of the usual method of sampling from the infinite-dimensional system.

There are existing studies on queueing dynamics on graphs. 
In \cite{sankararaman2019interference, banerjee2020ergodicity}, the authors study a queueing model on the finite-dimensional integer grids, with each node having a Poisson arrival and a service rate depending on the workload of neighboring queues under a processor sharing discipline. 
The interactions between neighboring nodes are not of the Jackson network type (nor of the mean-filed type); and the studies focus on stability and steady state analysis. 
In \cite{mandjes2018queues}, the authors study an open Markovian Jackson network of infinite-server queues where the connectivity between each pair of nodes alternates between up and down states. With the number of nodes being fixed, the authors study the transient behaviors of the queueing dynamics including the probability generating functions and diffusion approximations. In contrast to our paper, these models do not consider the behaviors of the networks as the number of nodes grows to infinity. 
Supermarket models on graphs are classical queueing networks in communication systems, where jobs are dispatched to parallel-server networks that may be interconnected via a graph, with the goal of achieving load balancing. 
The classical power-of-two policy is shown to be optimal in a finite system \cite{mitzenmacher2002power}, while there has recently been a stream of studies on large-scale systems as the number of nodes grows to infinity under various graph topologies; see e.g., \cite{mukherjee2018universality,mukherjee2018asymptotically,budhiraja2019supermarket,weng2020optimal,rutten2024mean,goldsztajn2025fluid} and the survey \cite{der2022scalable}. However, in these works, there is no Jackson-type routing, and no infinite-dimensional Skorokhod reflection mapping is involved. 

There has been previous work exploring reflection operators on processes over both time and uncountably infinite dimension $\RR^+\times[0,1]$. Introduced in \cite{nualart1992white} and further studied in several other papers (\cite{zambotti2004occupation,dalang2006hitting,debussche2007conservative,donati1993white,zhang2010white,hambly2019reflected,dalang2013holder}), the work studies solutions of the heat equation on the spatial interval $[0,1]$, with coordinatewise normal reflection to guarantee that solutions remain positive. While we are also working in the spatial interval $[0,1]$, in our work the oblique reflection term is much more general and in fact acts on the system through the reflection operator ${\bf1}-\boldsymbol{F}$.

\subsection{Organization of the Paper}
The paper is organized as follows. In the next subsection,  we summarize the notation and some definitions used throughout the paper. In Section \ref{sec:model}, we first introduce the finite Jackson network model and define the queueing process, and  present the scaling assumptions on the parameters and processes. We then state the main results and discuss their implications. 
In Section \ref{sec:skhmap}, we develop the general theory for infinte-dimensional Skorokhod problems, proving theorems on existence, uniqueness, complementarity, and Lipschitz continuity. In Section \ref{sec:convproof}, we prove the convergence theorems given in Section \ref{sec:model}. Supplementary theory and additional proofs are given in Appendix \ref{sec:appendix}.

\subsection{Definitions and Notation}
We summarize the commonly used notation here. 
Let $\mathcal{D}_T^k = D([0,T], \mathbb{R}^k)$ be the space of c{\`a}dl{\`a}g functions on $[0,T]$ taking values in $\mathbb{R}^k$. For $x\in\mathcal{D}_T^k$, $x$ has components $x_i(t)$, $t\in[0,T]$, $i=1,\dots, k$. 
Let $\mathcal{C}_T^k$ be the subspace of continuous function in $\mathcal{D}_T^k$. For $k=1$, we simply write $\mathcal{C}_T$ and $\mathcal{D}_T$. 
For $x: [a,b] \to \RR$, we define $\|x\|_1 = \int_{[a,b]} |x(u)|du$

{\it Space  $\mathcal{D}_T(L_1)$.} We say $f \coloneq \{f_u(t):u \in [0,1], t \in [0,T]\} \in \mathcal{D}_T(L_1) $ if $f \in D([0,T], L_1[0,1])$, meaning that (i)
$\|f\|_{T,1}<\infty$, where 
$$\|f\|_{T,1}:= \int_0^1\sup_{0 \leq t \leq T} |f_u(t)|\,du,$$ and (ii) for each $u\in[0,1]$, 
$f_u \in \mathcal{D}_T$. 
Note that the dominated convergence theorem guarantees that $[0,T] \ni t \mapsto \{f_u(t) : u\in[0,1]\} \in L_1$ is c{\`a}dl{\`a}g for $f \in \mathcal{D}_T(L_1)$. 
We write $\mathcal{C}_T(L_1)$ to denote the subspace of $\mathcal{D}_T(L_1)$, containing functions that are continuous in $t$, that is, in condition (ii) above,
$f_u \in \mathcal{C}_T$. 
Define $\mathcal{D}_T^\uparrow(L_1) = \{f \in \mathcal{D}_T(L_1): f_u(0) \geq 0, \, df_u(t) \geq 0 \, \forall t\in [0,T], u \in [0,1] \}$, and define $\mathcal{C}^\uparrow_T(L_1)$ analogously. 
For $f,g\in\mathcal{D}_T(L_1)$, we say $f\geq g$ if $f_u(t)\geq g_u(t)$ for all $u\in[0,1]$ and $t\in[0,T]$. 
For $y\in\mathcal{D}_T(L_1)$ and any $t\in[0,T]$, $[f(t)]^+$ is defined for any $u\in[0,1]$ by $[f(t)]^+_u \coloneq [f_u(t)]^+$, where  $[x]^+ \coloneq \max(x,0)$ for $x\in\mathbb{R}$. 

{\it Operators.} For a function $G \colon [0,1]^2 \to \RR_+$, we use the bold symbol to denote the corresponding integral operator $\boldsymbol{G}$ by 
\begin{equation}
\label{eq:Operator}
    (\boldsymbol{G} f)(u) := \int_0^1 G(u,v)f(v)\,dv, \quad f:[0,1] \mapsto\mathbb{R}. 
\end{equation}
When an operator $\boldsymbol{G}$ is given, it should be understood that the non-bold $G$ represents the corresponding kernel function. In the case where we are dealing with functions in $\mathcal{D}_T(L_1)$, it should be understood that the operator acts on the spatial dimension, not the temporal, and can be written as
\begin{equation}
\label{eq:Operator2}
    (\boldsymbol{G} f)_u(t) := \int_0^1 G(u,v)f_v(t)\,dv,\quad f \in \mathcal{D}_T(L_1).
\end{equation}
Operator $\boldsymbol{G}$ is said to be \textit{nonnegative} and we write $\boldsymbol{G} \ge 0$ if $G(u,v)\geq0$ for all $(u,v) \in [0,1] \times [0,1]$.
The transpose $\boldsymbol{G}^T$ of operator $\boldsymbol{G}$ is defined as
\begin{equation}
\label{eq:OperatorTranspose}
    (\boldsymbol{G}^T f)(u) := \int_0^1 G(v,u)f(v)\,dv.
\end{equation}
The $n$-fold composition of operator $\boldsymbol{G}$ is denoted by $\boldsymbol{G}^{(n)}$. 
The operator norm $\|\boldsymbol{G}\|_{op}$ of an integral operator is given by
$$\|\boldsymbol{G}\|_{op} \coloneqq \sup_{v \in [0,1]}\int_0^1  |G(u,v)|du. $$
For integral operators $\boldsymbol{F}_1, \boldsymbol{F}_2$ and $f\in\mathcal{D}_T(L_1)$, we have $\|\boldsymbol{F}_1f\|_{T,1}\le \|\boldsymbol{F}_1\|_{op}\|f\|_{T,1}$ and $\|\boldsymbol{F}_1\boldsymbol{F}_2\|_{op}\le\|\boldsymbol{F}_1\|_{op}\|\boldsymbol{F}_2\|_{op}$. These properties are proved in Lemmas \ref{lem:opnorm} and \ref{lem:compOpNorm} of Appendix \ref{app:useful}.
The identity operator ${\bf1}$ acts on either $L_1$ or $\mathcal{D}_T(L_1)$ through 
\begin{align*}
(\boldsymbol{1}f)(u) := f(u), \quad f:[0,1] \mapsto\mathbb{R}, \\
(\boldsymbol{1}f)_u(t) := f_u(t), \quad f \in \mathcal{D}_T(L_1).
\end{align*}
As $\|\boldsymbol{1}f\|_{T,1} = \|f\|_{T,1}$, we also define $\|{\bf1}\|_{op}=1$.
For functions $f,g \in \mathcal{D}_T(L_1)$, we define $\operatorname{diag}(f)$ by
$(\operatorname{diag}(f)g)_u(t) = f_u(t)g_u(t)$.

Let $\boldsymbol{F}$ be an operator as defined in \eqref{eq:Operator}. 
We say $\boldsymbol{F}\in\mathcal{R}$ if $F\geq0$, $\|\boldsymbol{F}\|_{op}\leq1$, and $\rho(\boldsymbol{F})<1$, where $\rho(\boldsymbol{F})$ denotes the spectral radius of $\boldsymbol{F}$. We call $\mathcal{R}$ the set of reflection operators. 

Given a Polish space $\mathbb{S}$, denote by $\mathcal{P}(\mathbb{S})$ the space of probability measures on $\mathbb{S}$ endowed with the topology of weak convergence, so that it is also a Polish space.
Denote by $W_1$ the Wasserstein-$1$ metric on $\mathcal{P}(\mathcal{D}_T)$, that is,
\begin{equation*}
    W_1(\mu,\nu) := \sup_{f} \left(\int_{\mathcal{D}_T} f \,d\mu - \int_{\mathcal{D}_T} f \,d\nu \right), \quad \mu,\nu \in \mathcal{P}(\mathcal{D}_T),
\end{equation*}
where the supremum is taken over all functions $f \colon \mathcal{D}_T \to \mathbb{R}$ that is $1$-Lipschitz.

\medskip

\section{Model and Main Results}
\label{sec:model}

\subsection{Open Jackson Network Model}
Consider an open Jackson network with $N$ nodes/stations. Jobs can arrive and leave externally while also circulating in the network. Let $A^N_i =\{A^N_i(t): t\ge 0\}$ be the external arrival process at station $i$, $i=1,\dots,N$, which are mutually independent Poisson processes with rates $\lambda_i^N$. The service times in each node are independent and exponentially distributed, with service rate $\mu_i^N$ at node $i=1,\dots,N$. 
The transition probabilities  $p_{ij}^N$ between nodes in the network together with the exit probabilities $p_{i0}^N$ satisfy
$$\sum_{j=1}^N p_{ij}^N = 1 - p_{i0}^N \in [0,1].$$
Note that for open Jackson networks, there exists at least one $i$ such that  $p_{i0}^N>0$.  Denote $P^N := (p^N_{ij})_{i,j=1,\dots,N}$.

Let $Q_i^N = \{Q_i^N(t): t\ge 0)\}$ be the queue length at station $i$, $i=1,\dots,N$. 
Then we can write the dynamics of $Q_i^N$ as
\begin{align} \label{eqn-QiN-closed}
Q_i^N(t) = Q_i^N(0) + A^N_i(t) + \sum_{j=1}^N  \Poisson^N_{ji}(\mu_j^N B_j^N(t)) - \sum_{j=0}^N \Poisson^N_{ij}(\mu_i^N B_i^N(t))
\end{align}
where $\Poisson^N_{ij}=\{\Poisson^N_{ij}(t): t\ge 0\}$, $i=1,\dotsc,N$, $j=0,\dotsc,N$, are mutually independent, rate-$p_{ij}^N$ Poisson processes, and $B_i^N = \{B_i^N(t): t\ge 0\}$ is the cumulative busy time process at station $i$, $i=1,\dots,N$, that is,
\[B^N_i(t) = \int_0^t 1_{\{Q^N_i(s)>0\}} ds, \quad t \ge 0, \, i =1,\dots,N. \]
From this we also have the cumulative idle time $I^N_i(t) = t- B^N_i(t)$ satisfying
\begin{equation}
\label{eq:barINi}
\int_0^\infty 1_{\{Q^N_i(t)>0\}} dI^N_i(t)=0, \quad \forall\,i=1,\dots,N.
\end{equation}

We denote $Q^N = \{Q^N_1,\dots,Q^N_N\}$ and similarly for $B^N$ and $I^N$. It is clear that these processes have sample paths in $\mathcal{D}_T^N$.

In this paper, we consider a large-scale Jackson network model, where the number of nodes/stations $N$ grows to infinity. 
Note that there is implicitly an underlying network topology, that is, a graph with the nodes being the service stations/servers and the edges indicating the connectivity of the nodes (represented as in the routing scheme). The underlying graph is not necessarily complete, but we assume that it is dense (that is, $p_{ij}^N = O(1/N)$ for $j \ne i$ and hence the degree of each note $i$ is $O(N)$). This rules out certain Jackson network models such as tandem networks and star-shaped networks, since they are not dense. See the formal assumptions on the scaling of $P^N$ in the next subsection and also the discussions in Remark \ref{rem:Pii}.

\subsection{Scaling Assumptions } 

We let 
\begin{equation} \label{eq:p-G-def}
p_{ij}^N =  \frac{1}{N} G_{ij}^N, \quad i,j=1,\dots,N,
\end{equation}
where $G^N=(G^N_{ij})$ satisfies
\begin{equation} \label{eq:GNsumN}
\sum_{j=1}^N G_{ij}^N \leq N.
\end{equation}
We will impose a series of assumptions on $G^N$ in the following. 

Abusing notation, we define  $G^N$ as a block-wise constant function defined on $[0,1] \times [0,1]$:
\begin{equation} \label{eqn-GN-uv}
G^N(u,v) = \sum_{i=1}^N \sum_{j=1}^N G^N_{ij} {\bf1}_{u \in K^N_i} {\bf1}_{v\in K^N_j},
\end{equation}
where $K^N_i, i=1,\dots, N$, is a partition of $[0,1]$, satisfying $\int_0^1{\bf1}_{v\in K^N_i}dv=|K^N_i|=\frac{1}{N}$. For simplicity we assume that $K^N_1 = [0,1/N]$ and $K^N_i=(\frac{i-1}{N},\frac{i}{N}]$ for $2 \le i \le N$. We denote $\boldsymbol{G}^N$ as the operator corresponding to $G^N:[0,1]^2\to \RR_+$, as defined in \eqref{eq:Operator}, and also $(\boldsymbol{G}^N)^T$ its transpose. 

We first assume the following.
\begin{assumption} \label{assumption:GNGop}
For each $N$ we have $\rho(P^N)<1$, or equivalently, $\rho(G^N)<N$. Furthermore, 
there exists some $G:[0,1]^2\to \RR_+$ with corresponding operator $\boldsymbol{G}$ such that for some constants $C>0$ and $\alpha >0$,  
\begin{equation}
\label{eq:GntoG}
\|(\boldsymbol{G}^N)^T-\boldsymbol{G}^T\|_{op} \leq CN^{-\alpha/2},
\end{equation}
where $\boldsymbol{G}^T \in\mathcal{R}$, that is, the spectral radius $\rho(\boldsymbol{G}^T)<1$  and $\|\boldsymbol{G}^T\|_{op}\le 1$, or equivalently for all $u\in[0,1]$, 
\begin{equation}
\label{eq:G-int1}
\int_0^1 G(u,v)\,dv \leq 1.
\end{equation}
\end{assumption}

The following lemma guarantees that the blockwise constant operator associated with the transition matrix satisfies the same conditions as the limiting object. Its proof is given in Appendix \ref{sec:additional-pf}

\begin{lemma} \label{lem:GNinR}
Under Assumption \ref{assumption:GNGop}, we have $(\boldsymbol{G}^N)^T\in\mathcal{R}$.
\end{lemma}

\begin{remark}
Under equation \eqref{eq:GntoG} it follows that for sufficiently large $N$ we will have $\rho(P^N)<1$ and therefore $\rho((\boldsymbol{G}^N)^T)<1$ so that our later convergence results still hold. However, for ease of exposition, we directly assume $\rho(P^N)<1$ for all $N$.
\end{remark}

\begin{remark}
For Jackson networks, the transition probability from node $i$ to node $j$ may not be equal to that from node $j$ to node $i$. Hence, in our setting, $G$ may be asymmetic.  
Since graphon kernel functions are usually assumed to be symmetric and the so-called cut metric is necessarily used (see \cite{lovasz2012large}),  
 we refrain from calling the function $G$ a graphon kernel function.  

We also remark that in the graphon particle system of diffusions, e.g.\ in \cite{BayraktarChakrabortyWu2023graphon}, the interaction appears in the drift in the form $\frac{1}{N}\sum_{j=1}^N \xi_{ij}^N b(X_i^N,X_j^N)$, and similarly in the diffusion coefficient. Here, $\xi_{ij}^N\in[0,1]$ represents the connectivity between particles $i$ and $j$. Our current setup treats $G_{ij}^N$ as a kernel, which is not necessarily bounded by $1$, but instead satisfies \eqref{eq:GNsumN}.
\end{remark}

\begin{remark}
    Here we start from a finite system with $G^N$ and assume that it will converge to a limiting kernel function $G$ as given in Assumption \ref{assumption:GNGop}.
    One may also start from a kernel function $G$, and construct $G^N$ by a random sampling approach as done in  the literature of graphon particle systems (see, e.g., \cite{BayraktarChakrabortyWu2023graphon,BetCoppiniNardi2020weakly,Lucon2020quenched,CoppiniDeCrescenzoPham2024nonlinear}).
    This gives a way to construct a sequence of growing network models consistently in the statistical sense.
    We refer to \cite{Whitt1984open} 
    for a different way of constructing consistently growing closed Jackson networks in the sparse setting. 
\end{remark}

\begin{remark}  \label{rem:Pii}
    Assumption \ref{assumption:GNGop} on the scaling of $\{p_{ij}^N\}$ (as defined in \eqref{eq:p-G-def}) rules out the cases where $p_{ii}^N = O(1)$ for certain $i$'s.
    For example, if $p_{ii}^N \equiv 0.7$ and $p_{ij}^N = \frac{0.3}{N-1}$, then one would expect $G(u,v) \equiv 0.3$ for any $u\neq v$ and still have $\int_0^1 G(u,v)dv\le 1$ for any value of $G(u,u) \in \RR_+$. In this case, however, we do not have convergence of $\boldsymbol{G}^N$ to $\boldsymbol{G}$ as specified in the assumption. 
    This goes beyond the scope of the current work, and we leave it as a future work.
\end{remark}

\subsection{Convergence}

We define the following scaled queueing processes (with similar definitions for $\bar{B}^N_i$, $\bar{I}^N_i$ and $\bar{A}^N_i$):
\begin{equation}
\label{eq: scaledprocess}
    \bar{Q}_i^N(t) := 
\frac{Q_i^N(N^\alpha t)}{N^\alpha}, \quad t\ge 0.
\end{equation}
Recall that $\alpha>0$ already appears in Assumption \ref{assumption:GNGop}.
We also define 
\begin{equation}\label{initQinft}
 \bar{Q}_u^N(0) =   \sum_{i=1}^N \bar{Q}_i^N(0)  {\bf1}_{u \in K^N_i}, \quad u \in [0,1].
\end{equation}
Similarly, define the quantities $\lambda^N_u, \mu^N_u$ for each $u \in [0,1]$, that is,
\[
\lambda^N_u = \sum_{i=1}^N \lambda^N_i  {\bf1}_{u \in K^N_i}, \quad \mu^N_u = \sum_{i=1}^N \mu^N_i  {\bf1}_{u \in K^N_i}, \quad u \in [0,1].
\]

We make the following assumption on the primitives $\bar{Q}^N_u(0)$, $\lambda^N_u$ and $\mu^N_u$. 

\begin{assumption}
\label{assumption:Nbound}
There exist deterministic and nonnegative functions $\bar{Q}_u(0)$, $\lambda_u$ and $\mu_u$  in $L_1[0,1]$ such that $\mu_u>0$ for all $u \in [0,1]$, and 
for some $C>0$, 
$$\max\big\{\|\bar{Q}^N(0)-\bar{Q}(0)\|_1, \|\lambda^N-\lambda\|_1, \|\mu^N-\mu\|_1\big\} \leq C N^{-\alpha/2}.$$
\end{assumption}

\begin{definition}  \label{def:GFM}
The infinite-dimensional fluid model is defined as follows: 
\begin{equation}
    \label{eq:Qbar-new}
    \bar{Q}(t) =    \bar{X}(t) + ({\bf1}-\boldsymbol{G}^T) \bar{Y}(t),
\end{equation}
where 
\begin{equation}
    \label{eq:Xbar-new}
  \bar{X}(t) =\bar{Q}(0) + \big[\lambda  - ({\bf1} - \boldsymbol{G}^T)\mu \big]  t,
\end{equation}
\begin{equation}
    \label{eq:Ybar-new}
\bar{Y}(t)=\operatorname{diag}(\mu) \bar{I}(t),
\end{equation}
and $ \bar{I}(t)$ satisfies 
\begin{equation}
\label{eq:idletime}
d\bar{I}_u(t)\geq0\,, \quad \bar{I}_u(0) = 0\,, \mbox{ and } \int_0^\infty 1_{\{\bar{Q}_u(s)>0\}} \bar{I}_u(ds)=0, \quad \forall\,u \in [0,1].
\end{equation}
We denote $\bar{Q} =\Phi_{\boldsymbol{G}^T}(\bar{X}) $ and $\bar{Y} =\Psi_{\boldsymbol{G}^T}(\bar{X})$ and call them as the reflected process and regulator, respectively. It is apparent that $\bar{Q}(0)\ge0$ and $\bar{X}\in\mathcal{C}_T(L_1)$. 
\end{definition}

Equation \eqref{eq:Qbar-new} can be explicitly written as
\begin{equation}
    \label{eq:Qbar}
    \bar{Q}_u(t) = \bar{X}_u(t) + \left( \mu_u\bar{I}_u(t) - \int_0^1 \mu_v \bar{I}_v(t) G(v,u) \,dv \right),
\end{equation}
where 
\begin{equation}\label{eq:Xbar}
\bar{X}_u(t) = \bar{Q}_u(0) + \lambda_u t -\left(\mu_u  - \int_0^1 \mu_v  G(v,u) \,dv \right)t. 
\end{equation}
Equivalently, the above can be rewritten as 
\begin{equation}
     \label{eq:Qbar1}
    \bar{Q}_u(t) = \bar{Q}_u(0) + \lambda_u t - \mu_u (t - \bar{I}_u(t)) + \int_0^1 \mu_v \left(t-\bar{I}_v(t)\right) G(v,u) \,dv.
\end{equation}

\begin{prop}
    \label{prop:Lipschitz}
Under the conditions on $\boldsymbol{G}$ in Assumption \ref{assumption:GNGop}, the infinite-dimensional fluid limit in Definition \ref{def:GFM} has a unique solution $(\bar{Q}, \bar{Y}) \in \mathcal{C}_T(L_1) \times \mathcal{C}^\uparrow_T(L_1)$ given  $\bar{X}$.  
Moreover, 
the mappings $\Phi_{\boldsymbol{G}^T}$ and $\Psi_{\boldsymbol{G}^T}$ are Lipschitz continuous. 
\end{prop}

\begin{proof}
The existence and uniqueness result follows from Theorem \ref{thm:summary}. Since $\bar{X}\in\mathcal{C}_T(L_1)$, $(\bar{Q}, \bar{Y}) \in \mathcal{C}_T(L_1) \times \mathcal{C}^\uparrow_T(L_1)$ follows from Lemma \ref{lem:contXcontY}. Finally, the Lipschitz continuity property follows from Theorem \ref{thm:lipschitz}.  
\end{proof}

The following Theorem and Corollaries give convergence results from the finite system to its infinite dimensional approximation. Their proofs are given in Section \ref{sec:convproof}.

\begin{theorem} \label{thm-main}
Under Assumptions \ref{assumption:GNGop} and \ref{assumption:Nbound}, 
\begin{equation} \label{eqn:main1}
\frac{1}{N} \sum_{i=1}^N \E \sup_{0 \le t \le T} \bigg|\bar{Q}_i^N(t)-N\int_{(i-1)/N}^{i/N}\bar{Q}_u(t)\,du\bigg|\leq \kappa_T N^{-\alpha/2},
\end{equation}
for some constant $\kappa_T>0$. 
\end{theorem}

\begin{coro} \label{coro:bQNbQ-conv0}
    Suppose that 
    the set of discontinuities of the mapping $[0,1] \ni u \mapsto \{\bar{Q}_u(t) : t \in [0,T]\} \in \mathcal{C}_T$ has Lebesgue measure $0$, $\|\boldsymbol{G}\|_{op}<\infty$, and that $\lambda$ and $\mu$ are bounded. Then, under the same assumptions in Theorem \ref{thm-main},
    \begin{equation}
        \label{eq:cvg-coupled}
        \frac{1}{N} \sum_{i=1}^N \E \sup_{0 \le t \le T} |\bar{Q}_i^N(t)-\bar{Q}_{i/N}(t)| \to 0, \quad \mbox{ as } N \to \infty.
    \end{equation}
\end{coro}

In the following Corollary, we give an example for the first condition in Corollary \ref{coro:bQNbQ-conv0} to hold.

\begin{coro}
    \label{coro:example}
    Suppose $\bar{Q}_u(0),\lambda_u,\mu_u$ are bounded and continuous in a.e.\ $u$ and $G(u,v)$ is continuous in a.e.\ $(u,v)$.
    Also suppose $\|\boldsymbol{G}\|_{op}<\infty$.
    Then, under the same assumptions in Theorem \ref{thm-main}, \eqref{eq:cvg-coupled} holds.
\end{coro}

The assumption of a.e.\ continuity of $G$ in Corollary \ref{coro:example} can be satisfied, for example, by block networks and nearest-neighbor networks; see examples (c) and (e) from Remark \ref{rmk:examples}.

\bigskip

We also consider the empirical measure $\nu^N$ and the averaged measure $\bar\nu$:
$$\nu^N= \frac{1}{N} \sum_{i=1}^N \delta_{\bar{Q}^N_i(\cdot)}, \quad \bar\nu:= \int_0^1 \delta_{\bar{Q}_u(\cdot)}\,du,$$
where $\delta_x$ denotes the Dirac measure at the point $x$. We present the following convergence result, with the proof given in Section  \ref{subsec:empconv}.

\begin{theorem} \label{thm:main2}
Under the same assumptions in Theorem \ref{thm-main}, we have
\begin{equation} \label{eq:nuN-conv}
W_1(\nu^N,\bar\nu) \to 0 \quad \mbox{in probability as } N \to \infty. 
\end{equation}
Moreover, for any function $f\colon \mathcal{D}_T\to \RR$ that is Lipschitz with respect to the uniform metric (that is, $|f(x)-f(y)| \le C\sup_{0 \le t \le T} |x(t)-y(t)|$ for all $x,y \in \mathcal{D}_T$ for some $C \in (0,\infty)$), we have
\begin{equation} \label{eq:nuNf-conv}
\frac{1}{N} \sum_{i=1}^N f(\bar{Q}^N_i)\to \int_0^1 f(\bar{Q}_u)\,du \quad \mbox{ in probability as } N \to \infty.
\end{equation}
\end{theorem}

\begin{remark}
    \label{rem:functionalmeasure}
    We first remark that the function $f$ is over the paths of $\{\bar{Q}^N_i(t): t \in [0,T]\}$ and the limit $\{\bar{Q}_u(t): t \in [0,T]\}$ for each $u \in [0,1]$.
    Also, \eqref{eq:nuN-conv} automatically implies \eqref{eq:nuNf-conv} for bounded and continuous functions $f$ on $\mathcal{D}_T$, and Theorem \ref{thm:main2} shows that \eqref{eq:nuNf-conv} also holds for Lipschitz functions.
    We discuss a few examples of applications with corresponding functionals of the paths below.
    \begin{enumerate}[label=(\alph*)]
    \item 
        The running maxima, corresponding to the maximum of the queue length over $[0,T]$.
        Let $f \colon \mathcal{D}_T \to \mathbb{R}$ be given by
        \begin{equation*}
        f(q) = h\Big(\sup_{0 \le t \le T} q(t)\Big), \quad q \in \mathcal{D}_T,
        \end{equation*}
        where $h$ is a Lipschitz cost function, such as $h(x) = (x-a)^+$ for some $a \in (0,\infty)$. 
        Since $f$ is Lipschitz, from \eqref{eq:nuNf-conv} we have
        \begin{equation*}
            \frac{1}{N} \sum_{i=1}^N h\Big(\sup_{0 \le t \le T}\bar{Q}^N_i(t)\Big) \to \int_0^1 h\Big(\sup_{0 \le t \le T}\bar{Q}_u(t)\Big)\,du
        \end{equation*}
        in probability as $N \to \infty$.
    
    \item 
        The path integration, corresponding to the total waiting time over $[0,T]$.
        Let $f \colon \mathcal{D}_T \to \mathbb{R}$ be given by
        $$f(q)=\int_0^T q(t)\,dt, \quad q \in \mathcal{D}_T.$$
        Since $f$ is Lipschitz, from \eqref{eq:nuNf-conv} we have
        \begin{equation}
            \label{eqn-int-conv}
            \frac{1}{N} \sum_{i=1}^N \int_0^T \bar{Q}^N_i(t)\,dt \to \int_0^1 \Big( \int_0^T \bar{Q}_u(t)\,dt \Big)\,du
        \end{equation}
        in probability as $N \to \infty$.
        Note that this means
        \begin{equation*}
            \int_0^T \frac{1}{N} \sum_{i=1}^N \bar{Q}^N_i(t)\,dt \to \int_0^T \Big( \int_0^1 \bar{Q}_u(t)\,du \Big)\,dt,
        \end{equation*}
        which essentially says the convergence of the average queue length.
        
    \item 
        Projection map.
        Let $f \colon \mathcal{D}_T \to \mathbb{R}$ be given by 
        $$f(q)=q(t), \quad q \in \mathcal{D}_T,$$
        for some fixed $t \in [0,T]$.
        Since $f$ is Lipschitz, from \eqref{eq:nuNf-conv} we have
        \[
        \lim_{N\to\infty}\frac{1}{N} \sum_{i=1}^N  \bar{Q}^N_i(t)  =\int_0^1 \bar{Q}_u(t)\,du \quad \text{in probability}. 
        \] 
    \end{enumerate}
\end{remark}

\begin{remark}
\label{rmk:examples}
Our model formulation allows various forms of heterogeneity, in terms of the arrival and service rates at each station and the transition probabilities (equivalently, the underlying network topology). 
We first observe from \eqref{eq:Qbar}--\eqref{eq:Xbar}, that, if the arrival and service rates are homogeneous, that is, $\lambda_u \equiv \lambda$ and $\mu_u \equiv \mu$ for constant $\lambda, \mu$, assuming that $g_u:=\int_0^1 G(v,u) \,dv$ is also a constant, denoted as $g$, and in addition, $\bar{Q}_u(0) \equiv \bar{Q}(0)$ for all $u$, then we obtain a homogeneous ODE (in fact, a standard mean-field limit). The quantity $g_u:=\int_0^1 G(v,u) \,dv$ being constant means the total routing probabilities inflow for each node are the same, however, it does allow certain heterogeneity in the connectivity since we are not requiring $G(u,v)$ to be constant. 

On the other hand, our model allows for many choices of the kernel $G(u,v)$. We give a few examples below:
\begin{enumerate}
    \item[(a)] Symmetric network model: the transition probabilities for the finite model satisfy $p^N_{ij}= p^N_{ji}$ for each $i,j$, and in the limit $G(u,v) = G(v,u)$ for each $u, v \in [0,1]$.
If, in addition, $\lambda_u\equiv \lambda$ and $\mu_u \equiv \mu$ for all $u \in [0,1]$, then we have $\bar{X}_u(t) = \bar{Q}_u(0) + (\lambda  -\mu G_0(u)) t$ where $G_0(u):=1  - \int_0^1  G(u,v) \,dv$. Note that $G_0(u)$ can be interpreted as the probability of leaving the network after service at node ``$u$" in the limit. Then we obtain
$\bar{Q}_u(t) = \bar{Q}_u(0) + (\lambda  -\mu G_0(u)) t + \mu \left( \bar{I}_u(t) - \int_0^1  \bar{I}_v(t) G(u,v) \,dv \right).$

\item[(b)] Bipartite network model: suppose we
partition the set of nodes into two subsets $\{1, \dots, K\}$ and $\{K+1,\dots,N\}$, and the transitions only occur across the two subsets but not within the subsets. In this case, the kernel function will take the form: for some $k\in(0,1)$ (corresponding to the limit of $K/N$), $G(u,v)=0$ for $(u,v)\in([0,k]\times[0,k])\cup ([k,1]\times[k,1])$, and 
$G(u,v)\ge 0$ can be arbitrary for $(u,v)\in([0,k]\times[k,1])\cup ([k,1]\times[0,k])$. 

\item [(c)] Block network model: suppose we have $\{J_i\}_{i=1,\cdots M}$ and $\{K_i\}_{i=1,\cdots M'}$ to be any two arbitrary fixed partitions of $[0,1]$, and 
$G(u,v)$ is defined as a constant over each rectangle $J_i\times K_j$.  
Given $N$ nodes, we can construct a network model $G^N$ by sampling, $G^N_{ij} = G(i/N,j/N)$. 

\item [(d)] Network model with clusters: For a finite network, let $\{R^N_l\}_{l=1,\cdots,M}$ be a partition of $\{1,\cdots,N\}$ in the increasing order. For $l\neq l'$, $p^N_{ij} \ll 1/N$ for $i\in R^N_l$ and $j\in R^N_{l'}$. Since $G^N_{ij} = Np^N_{ij}$, this corresponds to $G^N_{ij} \ll 1$ for $i\in R^N_l$ and $j\in R^N_{l'}$. Let $\{J_i\}_{i=1,\cdots M}$ be the corresponding partition of $[0,1]$. 
In the case where different clusters do not interact, we have $G(u,v)\geq0$ on $J_i \times J_i$, $i=1,\cdots,M$ and $G(u,v) = 0$ for $u\in J_i, v\in J_{i'}$ with $i\neq i'$.  One may also consider a case where $l\neq l'$, $p^N_{ij} = O(\epsilon)/N$ for $i\in R^N_l$ and $j\in R^N_{l'}$ and for small $\epsilon>0$, which leads to $G(u,v) = O(\epsilon)$ for  $u\in J_i, v\in J_{i'}$ with $i\neq i'$.

\item [(e)] Network with nearest dense $k$-neighbors on a circle: Let $k$ be an integer less than $n$ such that $k=O(N)$, i.e., $k=\left\lfloor\alpha N\right\rfloor$ for some $0<\alpha<1$ where we assume that $\alpha$ is fixed as $N\to\infty$ (assume $k$ even for simplicity). Each node is connected to its $k$ nearest neighbors through a routing matrix satisfying $p_{ij}=(1-\varepsilon)/k$ for $j\in[i-\frac{k}{2},\dots,i-1,i+1,\dots,i+\frac{k}{2}]\mbox{ mod } N$ and some $0 < \varepsilon < 1$. In the graph limit this corresponds to 
$G(u,v)=\frac{1-\varepsilon}{\alpha}\mathsf{1}(\min{(|u-v|,1-|u-v|)} \le \alpha)$.

\item[(f)] Other functionals: 
    $G(u,v) = au^\alpha v^\beta$, for $a,\alpha, \beta >0$, $u, v \in [0,1]$, satisfying $\|\boldsymbol{G}\|_{op}\le 1$,
    or 
    $G(u,v)=a+b\sin(2\pi u)+c\sin(2\pi v)$ for any $a, b, c$ satisfying $1\ge a>|b|+|c|$. 
\end{enumerate}

For any infinite system taking a form of the above the kernel $\boldsymbol{G}$ must be chosen to satisfy  $\boldsymbol{G}^T\in\mathcal{R}$.
In this case, convergence from the finite to infinite system can be guaranteed.
\end{remark}

\bigskip

\section{An Infinite-Dimensional Skorokhod Reflection Mapping}
\label{sec:skhmap}

In this section, we study an infinite-dimensional Skorokhod mapping, proving the existence and uniqueness of a solution, and the Lipschitz continuity property. 

\subsection{Motivation}
\label{FequalsGT}\ \\
If we write $Z(t)\coloneq \bar{Q}(t)$, $X(t) \coloneq \bar{Q}(0) + \lambda t - ({\bf1}-\boldsymbol{G}^T)\mu t$ and $Y(t) \coloneq \bar{Y}(t) \coloneq \operatorname{diag}(\mu)\bar{I}(t)$, from relations \eqref{eq:Qbar-new}--\eqref{eq:idletime}  and writing $\boldsymbol{F} \coloneq \boldsymbol{G}^T$, we have the following system:
\begin{equation} 
    \label{eq:cond1op}
    Z(t) = X(t) +  ({\bf1}-\boldsymbol{F}) Y(t) \geq 0,
\end{equation}
\begin{equation}
    \label{eq:cond2op}
    dY(t) \geq 0 \text{  and  } Y(0) = 0,
\end{equation}
\begin{equation} 
    \label{eq:cond3op}
    Z_u(t)dY_u(t) = 0 \text{ for all } u \in [0,1]\,.
\end{equation}

\begin{remark}
\label{rem:refregdef}
We say that this set of equations defines an infinite-dimensional Skorokhod mapping in $\mathcal{D}_T(L_1)$. 
When  we have existence and uniqueness of $(Z,Y)$ given $X$ and $\boldsymbol{F}$, denote by $\Phi_F(X):=Z$ and $\Psi_F(X):=Y$ the \textbf{reflected process} and \textbf{regulator}, respectively. 
\end{remark}

Given a function $X$ and an operator $\boldsymbol{F}$, one may ask if a solution to the system \eqref{eq:cond1op}--\eqref{eq:cond3op} exists, and if so whether such a solution is unique. The following theorem shows that such existence and uniqueness holds for suitable spaces.

\begin{theorem}
\label{thm:summary}
Take any $X\in\mathcal{D}_T(L_1)$ with $X(0)\ge0$ and $\boldsymbol{F}\in\mathcal{R}$. There is a unique solution $(Y,Z)\in\mathcal{D}_T^\uparrow(L_1)\times\mathcal{D}_T(L_1)$ to the system given by \eqref{eq:cond1op}--\eqref{eq:cond3op}.
\end{theorem}

\noindent We will prove the above theorem in the rest of this subsection. Existence is shown in Theorem \ref{thm:exrefmap}, uniqueness in Theorem \ref{thm:unqfixpoint}, and the complementarity condition \eqref{eq:cond3op} is shown in Theorem \ref{thm:compchar}.

\subsection{Definition and Characterization}

\begin{definition}[Reflection map in $\mathcal{D}_T(L_1)$]
Let $X$ be any function in $\mathcal{D}_T$. For any $\boldsymbol{F} \in \mathcal{R}$, define the \textit{feasible regulator} set $\mathbf{\Psi}_F$ by 
$$\mathbf{\Psi}_F(X) \coloneq \{W \in \mathcal{D}_T^\uparrow(L_1) : X+(\mathbf{1}-\boldsymbol{F})W \geq 0\}.$$
Given such $\mathbf{\Psi}_F$, further define the \textit{reflection map} $R_F \coloneq (\Psi_F,\Phi_F): \mathcal{D}_T(L_1) \rightarrow \mathcal{D}_T^\uparrow(L_1) \times \mathcal{D}_T(L_1)$ of $X$ by
\begin{align}
    & Y \equiv \Psi_F(X) \coloneq \inf \mathbf{\Psi}_F(X), \\
    & Z \equiv \Phi_F(X) \coloneq X + (\mathbf{1}-\boldsymbol{F})Y. 
\end{align}
Here the infimum over $\mathbf{\Psi}_F(X)$ is understood as the following. For each $u \in [0,1]$ and $t \in [0,T]$,
$$Y_u(t) \coloneq \inf\{W_u(t) \in \mathbb{R}: W \in \mathbf{\Psi}_F(X)\}$$
\end{definition}

As $\inf \mathbf{\Psi}_F(X)$ is taken componentwise, we still have to show that $\mathbf{\Psi}_F(X) \neq \emptyset$ and $Y \in \mathbf{\Psi}_F$, for which we will need the next two lemmas. As will be proved, the reflection map $R_F$ is precisely the reflected process and regulator of the system \eqref{eq:cond1op}--\eqref{eq:cond3op}.

\begin{lemma}
\label{eq:infimum}
Take $\mathcal{W} \subseteq\mathcal{D}_T^\uparrow(L_1)$. Then $\uW\, \coloneq \inf \mathcal{W} \in \mathcal{D}_T^\uparrow(L_1)$.
\end{lemma}
\begin{proof}
The property $\uW\,(0) \geq 0$ is immediate. Consider times $0 \leq t_1 < t_2 \leq T$ and fix $u \in [0,1]$.
Since $dW_u(t) \geq 0 \:\: \forall t\in[0,T]$, 
\begin{align*}
W_u(t_1) \leq W_u(t_2) \:\: \forall W\in \mathcal{W}. 
\end{align*}
Taking the infimum over $W \in \mathcal{W}$ gives
$$\uW_u(t_1) \leq \uW_u(t_2)$$
and so $d\uW_u(t)\geq 0 \:\: \forall t\in[0,T]$.

To show that $\uW\,$ is right-continuous, first note that for any increasing function $f$, $\lim_{h\downarrow0}f(x+h) = \inf_{h>0}f(x+h)$. Therefore, we have 
\begin{align*}
\lim_{h\downarrow0}\uW_u(t+h) &= \inf_{h>0}\uW_u(t+h) 
=\inf_{h>0}\inf_{W\in\mathcal{W}}W_u(t+h)\\
&=\inf_{W\in\mathcal{W}}\inf_{h>0}W_u(t+h)
=\inf_{W\in\mathcal{W}}W_u(t)
=\uW_u(t). 
\end{align*}
The property of left limits is immediate as $\uW_u(t)$ is increasing and bounded above by $W_u(t)$.

For each $t\in[0,T]$, for all $W\in\mathcal{W}$ the map $u\mapsto W_u(t)$ is in $L_1$. Since the infimum of measurable functions is itself measurable, so must be $\uW_u(t)$ in $u$. Noting that $0\leq\uW\,(t)\leq W(t) \le W(T)$, we have
$$\|\uW\,\|_{T,1} \leq \|W\|_{T,1}<\infty,$$
and hence $\uW\, \in \mathcal{D}_T^\uparrow(L_1)$.
\end{proof}

\begin{lemma}
\label{lem:maptospace}
For any $W\in \mathcal{D}_T^\uparrow(L_1)$
and $\boldsymbol{R}\in\mathcal{R}$, $\boldsymbol{R}W \in \mathcal{D}_T^\uparrow(L_1)$. 
\end{lemma}

\begin{proof}
Fix $u\in[0,1]$. Note that $(\boldsymbol{R}W)_u(0) = \int_0^1R(u,v)W_v(0)\,dv \geq 0$. For times $0\leq t_1<t_2\leq T$, 
$$(\boldsymbol{R}W)_u(t_1) = \int_0^1R(u,v)W_v(t_1)\,dv \leq \int_0^1R(u,v)W_v(t_2)\,dv = (\boldsymbol{R}W)_u(t_2).$$
Right continuity follows from
\begin{align*}
\lim_{h\downarrow0}(\boldsymbol{R}W)_u(t+h) &= \lim_{h\downarrow0}\int_0^1R(u,v)W_v(t+h)\,dv
=\int_0^1R(u,v)\lim_{h\downarrow0}W_v(t+h)\,dv\\
&=\int_0^1R(u,v)W_v(t)\,dv
=(\boldsymbol{R}W)_u(t),
\end{align*}
where the second equality follows from the dominated convergence theorem, noting that for $t+h\le T$, $|R(u,v)W_v(t+h)|\leq|R(u,v)W_v(T)|$. Left limits  exist by monotonicity. 

For any jump discontinuity at $t$, we have
$$(\boldsymbol{R}W)_u(t)-(\boldsymbol{R}W)_u(t-) = \int_0^1R(u,v)(W_v(t)-W_v(t-))\,dv\geq0. $$ 

From the properties of $R$, it follows that
$$\|\boldsymbol{R}W\|_{T,1} \le \|\boldsymbol{R}\|_{op} \|W\|_{T,1}<\infty.$$ 
Hence $\boldsymbol{R}W \in \mathcal{D}_T^\uparrow(L_1)$.
\end{proof}

The following lemma provides bounds on the $k$-fold operator of $\boldsymbol{F}$ and $(\mathbf{1}-\boldsymbol{F})^{-1}$, which will be used later to show existence and Lipschitz continuity of the reflection map.

\begin{lemma}
\label{lem: invop}
For $\boldsymbol{F} \in \mathcal{R}$ and any $\gamma\in(0,1)$, there exists $k\ge1$ such that $$\|\boldsymbol{F}^{(k)}\|_{op}\leq\gamma.$$ Furthermore, the inverse operator of $(\mathbf{1}-\boldsymbol{F})$ can be written as
$$(\mathbf{1}-\boldsymbol{F})^{-1} = \sum_{n=0}^\infty \boldsymbol{F}^{(n)}.$$
Furthermore, given such $\gamma$ and $k$, the following holds:
$$\|(\mathbf{1}-\boldsymbol{F})^{-1}\|_{op}\le \frac{k}{1-\gamma}\,.$$
\end{lemma}
\begin{proof}
Let $\|\cdot\|$ be any norm on $\boldsymbol{F}$. 
Since $\rho(\boldsymbol{F}) = \lim_{n\rightarrow\infty}\|\boldsymbol{F}^{(n)}\|^{\frac{1}{n}}$, for sufficiently large $n$ we have $\|\boldsymbol{F}^{(n)}\|\leq\delta^n$ for some $\delta\in(\rho(\boldsymbol{F}),1)$. Choosing $\|\cdot\|=\|\cdot\|_{op}$ proves the first statement. For the second statement, note that the partial sums $S_n = \sum_{k=0}^n \boldsymbol{F}^{(k)}$ form a Cauchy sequence, as for any $n>m$,
$$
\|S_n - S_m\| \leq \sum_{k=m+1}^n\|\boldsymbol{F}^{(k)}\| \leq \sum_{k=m+1}^\infty \|\boldsymbol{F}^{(k)}\| \leq \sum_{k=m+1}^\infty \delta^k \to 0 \quad \text{as }\,\, m \to \infty.
$$
Hence $S_n$ converges to $S=\sum_{n=0}^\infty \boldsymbol{F}^{(n)}$. Finally, since $(\mathbf{1}-\boldsymbol{F})S_n = \mathbf{1} - \boldsymbol{F}^{(n+1)}$, taking limits gives $(\mathbf{1}-\boldsymbol{F})S = \mathbf{1}$.
For the last part, mimicking the proof of \cite[Theorem 14.2.5]{Whitt2002}, 
we have 
\begin{equation*}
    \bigg\|\sum_{i=0}^{\infty} \boldsymbol{F}^{(i)}\bigg\|_{op} \le \sum_{i=0}^{\infty} \|\boldsymbol{F}^{(i)}\|_{op} \le \sum_{i=0}^{k-1} \|\boldsymbol{F}^{(i)}\|_{op} + \gamma \sum_{i=0}^{\infty} \|\boldsymbol{F}^{(i)}\|_{op}.
\end{equation*}
Since $\|\boldsymbol{F}^{(i)}\|_{op} \le 1$,  this means
\begin{equation*}
    \|(\mathbf{1}-\boldsymbol{F})^{-1}\|_{op} = \bigg\|\sum_{i=0}^{\infty} \boldsymbol{F}^{(i)}\bigg\|_{op}\le\sum_{i=0}^{\infty} \|\boldsymbol{F}^{(i)}\|_{op} \le \frac{k}{1-\gamma}\,.
\end{equation*}
\end{proof}

\begin{theorem}[Existence of the Reflection Map]
\label{thm:exrefmap}
For any $\boldsymbol{F}\in\mathcal{R}$, $X \in \mathcal{D}_T(L_1)$ and $t\in [0,T]$, let 
$$\widetilde{W}(t) \coloneq (\mathbf{1}-\boldsymbol{F})^{-1} \sup_{0\leq s \leq t}[-X(s)]^+.
$$
Then $\widetilde{W} \in \mathbf{\Psi}_F(X)$.
Furthermore,
$$Y \equiv \Psi_F(X) \in \mathbf{\Psi}_F(X).$$
\end{theorem}
\begin{proof}
From Lemma \ref{lem: invop}, $\widetilde{W}(t) = \sum_{n=0}^\infty (\boldsymbol{F}^{(n)}\,\, \sup_{0\leq s \leq t}[-X(s)]^+)$. 
Since $X \in \mathcal{D}_T(L_1)$, we have $\sup_{0\leq s \leq t}[-X(s)]^+ \in \mathcal{D}_T^\uparrow(L_1)$. 
It then follows from Lemma \ref{lem:maptospace} that $\widetilde{W} \in \mathcal{D}_T^\uparrow(L_1)$.

It is also clear that for any $t \in [0,T]$,
$$X(t) + (\mathbf{1}-\boldsymbol{F})\widetilde{W}(t) = X(t) + (\mathbf{1}-\boldsymbol{F})(\mathbf{1}-\boldsymbol{F})^{-1} \sup_{0\leq s \leq t}[-X(s)]^+ \geq X(t)-X(t) \geq 0.$$
Thus $\widetilde{W} \in \mathbf{\Psi}_F(X)$.

Since $\mathbf{\Psi}_F(X) \subseteq \mathcal{D}_T^\uparrow(L_1)$, Lemma \ref{eq:infimum} gives that $Y \in \mathcal{D}_T^\uparrow(L_1)$. We are now left to show that $X + (\mathbf{1}-\boldsymbol{F})Y \geq 0$. Fix $u\in [0,1]$ and $t\in [0,T]$. 
By definition of infimum, for fixed $u \in [0,1]$, $\forall \epsilon>0$, $\exists\, W \in \mathbf{\Psi}_F(X)$ such that $W_u(t) \leq Y_u(t) + \epsilon$.  Then, we have
\begin{align*}
    & (X+(\mathbf{1}-\boldsymbol{F})Y)_u(t) = X_u(t) + Y_u(t) - (\boldsymbol{F}Y)_u(t) \\
    & = X_u(t) + Y_u(t) - \int_0^1F(u,v)Y_v(t) \, dv \\
    & \geq X_u(t) + W_u(t) - \int_0^1F(u,v)Y_v(t) \, dv - \epsilon & \mbox{as $W_u(t) \leq Y_u(t) + \epsilon$}\\
    & \geq X_u(t) + W_u(t) - \int_0^1F(u,v)W_v(t) \, dv - \epsilon & \mbox{as $Y(t) \leq W(t)$ and $F(u,v) \geq 0$}\\ 
    & = X_u(t) + W_u(t) - (\boldsymbol{F}W)_u(t) - \epsilon \\ &= (X+(\mathbf{1}-\boldsymbol{F})W)_u(t) - \epsilon \geq -\epsilon & \mbox{as $W \in \mathbf{\Psi}_F(X)$}. 
\end{align*}
 As this holds for any $\epsilon > 0$ and all $u,t$, we conclude that $X+(\mathbf{1}-\boldsymbol{F})Y \geq 0$.
\end{proof}

\begin{definition}
\label{def:pidefD}
For a given $X\in \mathcal{D}_T(L_1)$ and $\boldsymbol{F}\in\mathcal{R}$, define the mapping $\pi_{X,F}:\mathcal{D}^\uparrow_T(L_1)\to\mathcal{D}^\uparrow_T(L_1)$ by
$$\pi_{X,F}(W)(t) = \sup_{0\leq s \leq t}[-X(s) + \boldsymbol{F}W(s)]^+.$$
More specifically, for each $u \in [0,1]$,
$$\pi_{X,F}(W)_u(t) = \sup_{0\leq s \leq t}[-X_u(s) + (\boldsymbol{F}W)_u(s)]^+\,.$$
\end{definition}

The next result 
proves the existence of a fixed point for functions in $\mathcal{D}_T^\uparrow(L_1)$ under the mapping $\pi_{X,F}$, which can be regarded as a  modified version of \cite[Theorem 3.8.1]{edwards2012functional} for the space $\mathcal{D}_T^\uparrow(L_1)$.  The proof is given in Appendix \ref{sec:fixed-point-existence-pf}. 

\begin{theorem}
    \label{thm:monotone-fixed-point}
    Suppose there exist elements $x^0, y^0 \in \mathcal{D}_T^\uparrow(L_1)$ such that
    $$x^0 \le y^0, \quad x^0 \le \pi_{X,F}(x^0), \quad \pi_{X,F}(y^0) \le y^0.$$
    Define $x^n$, $n\ge 1$, by
    $$x^{n+1}=\pi_{X,F}(x^n).$$
    Then there exists $x \in \mathcal{D}_T^\uparrow(L_1)$ such that
    $$x^n \uparrow x, \quad x = \pi_{X,F}(x).$$
\end{theorem}

\begin{lemma}
\label{lem:feasregineq}
The feasible regulator can be equivalently written as 
$$\mathbf{\Psi}_F(X) = \{W\in\mathcal{D}_T^\uparrow(L_1): W\geq\pi_{X,F}(W)\}.$$
\end{lemma}

\begin{proof}
First take $W\in\mathbf{\Psi}_F(X)$, then $W\geq-X+\boldsymbol{F}W$. As $W\in\mathcal{D}_T^\uparrow(L_1)$, $W(t) = \sup_{0\leq s \leq t}[W(s)]^+$ and so $W\geq\pi_{X,F}(W)$.
For the other direction, assume $W\geq\pi_{X,F}(W)$. Then $$W(t)\geq\sup_{0\leq s \leq t}[-X(s) + \boldsymbol{F}W(s)]^+\geq -X(t)+\boldsymbol{F}W(t)$$ and so it follows that $W\in\mathbf{\Psi}_F(X)$.
\end{proof}

\begin{theorem}[Uniqueness of the Reflection Map]
\label{thm:unqfixpoint}
    For a given $X\in\mathcal{D}_T(L_1)$ and $\boldsymbol{F}\in\mathcal{R}$, there exists a unique fixed point $W^*$ of the map $\pi_{X,F}$, that is,
    \begin{equation*}
        W^*_u = (\pi_{X,F}(W^*))_u, \quad \forall\,u\in[0,1].
    \end{equation*}
    Moreover, the regulator map $\Psi_F(X)$ is equal to the fixed point $W^*$, that is,
    \begin{equation*}
        (\Psi_F(X))_u = W^*_u, \quad \forall\,u\in[0,1].
    \end{equation*}
\end{theorem}

\begin{proof}
By Theorem \ref{thm:exrefmap}, $\exists\, W\in \mathbf{\Psi}_F(X)$.
By Lemma \ref{lem:feasregineq}, we have $\pi_{X,F}(W) \le W$. 
Clearly $0 \le W$ and $0 \le \pi_{X,F}(0)$.
Applying Theorem \ref{thm:monotone-fixed-point} with $x^0=0$ and $y^0=W$ gives the existence of a fixed point $W^* \in \mathcal{D}_T^\uparrow(L_1)$ such that $\pi^n_{X,F}(0) \uparrow W^*$ and $W^* = \pi_{X,F}(W^*)$.

Since $W^*\geq\pi_{X,F}(W^*)$, we also have $W^*\in\mathbf{\Psi}_F(X)$. 
By definition of $\Psi_F(X)$, clearly $0 \le \pi_{X,F}(\Psi_F(X)) \le \Psi_F(X)\leq W^*$. 
Therefore, 
$$\pi^n_{X,F}(0) \leq \pi^n_{X,F}(\Psi_F(X)) \leq \Psi_F(X) \quad \mbox{for all } n \in \mathbb{N}.$$
Letting $n\to\infty$ yields $W^*\leq\Psi_F(X)$ and hence $W^*=\Psi_F(X)$. 
\end{proof}

\begin{remark}
    Although $\pi_{X,F}$ is defined on $\mathcal{D}^\uparrow_T(L_1)$, the equality in Theorem \ref{thm:unqfixpoint} holds for every $u \in [0,1]$, that is, the uniqueness is understood in the sense of everywhere (instead of a.e.).
    That said, the uniqueness understood in the sense of a.e.\ still holds. 
    To see this, suppose $W \in \mathcal{D}^\uparrow_T(L_1)$ is such that $W_u = (\pi_{X,F}(W))_u$ for a.e.\ $u\in[0,1]$. 
    Then by modifying $W_u$ on a negligible set of $u$, one can get some $\widetilde{W} \in \mathcal{D}^\uparrow_T(L_1)$ such that $\widetilde{W}_u=W_u$ for a.e.\ $u\in[0,1]$ and $\widetilde{W}_u = (\pi_{X,F}(\widetilde{W}))_u \,\forall\, u\in[0,1]$.
    It then follows from the uniqueness in Theorem \ref{thm:unqfixpoint} that $W^*_u = \widetilde{W}_u \,\forall\, u\in[0,1]$ and hence $W^*_u=W_u$ for a.e.\ $u\in[0,1]$.
\end{remark}

\begin{theorem}[Complementarity Characterization] \label{thm:compchar}
Take $Y\in\mathbf{\Psi}_F(X)$ and let $Z = X+(\mathbf{1}-\boldsymbol{F})Y$. Then, 
$Y=\Psi_F(X)$ if and only if the pair $(Y,Z)$ satisfies the complementarity property:
$$\int_0^{\infty}Z_u(t)dY_u(t) = 0 \quad \mbox{for all }u\in[0,1]\,.$$
\end{theorem}

\begin{proof}
$(\implies)$ Suppose that $(Y,Z)$ fail to satisfy the complementarity property. 
Then there are some $t\in[0,T]$ and $u\in[0,1]$ such that $Z_u(t)>0$ and $Y_u(t)$ increases at $t$. 
We consider the following two cases: \\\\
\textbf{Case 1:} Suppose that $Y_u(t)>Y_u(t-)$. Then there must exist some $\epsilon>0$ and $\delta>0$ such that $Y_u(t) - Y_u(t-)>\epsilon$ and $Z_u(s)\geq\epsilon$ for all $t\leq s\leq t+\delta$. Define $\widetilde{Y}$ by
\begin{equation*}
\widetilde{Y}_u(s)=\begin{cases}
    Y_u(s), & 0\leq s< t,\\
    Y_u(s) - \epsilon, & t\leq s<t+\delta,\\
    Y_u(s) & t+\delta\leq s,
  \end{cases}
\end{equation*}
and let $\widetilde{Y}_v = Y_v$ for all $v\neq u$. Then $\widetilde{Y}$ is still in $\mathcal{D}_T^\uparrow(L_1)$.\\\\
\textbf{Case 2:} Suppose that $Y_u(t)=Y_u(t-)$. Then there must exist some $\epsilon>0$ and $\delta>0$ such that $Z_u(s)\geq\epsilon$ and $0< Y_u(s)-Y_u(t)\leq\epsilon$ for $t<s<t+\delta$. Define $\widetilde{Y}$ by
\begin{equation*}
\widetilde{Y}_u(s)=\begin{cases}
    Y_u(s), & 0\leq s\leq t,\\
    Y_u(t), & t< s<t+\delta,\\
    Y_u(s) & t+\delta\leq s,
  \end{cases}
\end{equation*}
and let $\widetilde{Y}_v = Y_v$ for all $v\neq u$. Clearly we still have $\widetilde{Y}\in\mathcal{D}_T^\uparrow(L_1)$.\\\\
In both of these cases, for $s\in[t,t+\delta)$,
\begin{align*}
(X + (\mathbf{1}-\boldsymbol{F})\widetilde{Y})_u(s) &= X_u(s) + ((\mathbf{1}-\boldsymbol{F})Y)_u(s) + ((\mathbf{1}-\boldsymbol{F})(\widetilde{Y}-Y))_u(s) \\
&= Z_u(s) + (\widetilde{Y}_u(s)-Y_u(s)) + (\boldsymbol{F}(Y-\widetilde{Y}))_u(s) \\
&\geq \epsilon - \epsilon + 0 =0,
\end{align*}
and for $v\neq u$,
\begin{align*}
(X + (\mathbf{1}-\boldsymbol{F})\widetilde{Y})_v(s) &= X_v(s) + ((\mathbf{1}-\boldsymbol{F})Y)_v(s) + ((\mathbf{1}-\boldsymbol{F})(\widetilde{Y}-Y))_v(s) \\
&= Z_v(s) + 0 + (\boldsymbol{F}(Y-\widetilde{Y}))_u(s) \\
&\geq 0.
\end{align*}
So $X+(\mathbf{1}-\boldsymbol{F})\widetilde{Y}\geq0$ and thus $\widetilde{Y}\in\mathbf{\Psi}_F(X)$. However, $\widetilde{Y}\leq Y$ and $\widetilde{Y}_u(s)<Y_u(s)$ for $s\in[t,t+\delta)$, contradicting minimality of $Y$.\\\\
$(\impliedby)$ Assume $(Y,Z)$ satisfies the complementarity property and write $\hat{Y} = \pi_{X,F}(Y)$. By Theorem \ref{thm:unqfixpoint}, we need only to show that $\hat{Y} = Y$. Since $Y\in\mathbf{\Psi}_F(X)$, by Lemma \ref{lem:feasregineq} we have $Y\geq\hat{Y}$. Suppose that there exist some $u$ and $t$ such that $Y_u(t)>\hat{Y}_u(t)$. Then there must also exist some $t_0$ such that $Y_u(t_0)>\hat{Y}_u(t_0)$ and $dY_u(t_0)>0$. 
But then
$$Z_u(t_0) = Y_u(t_0)-(-X_u(t_0)+(\boldsymbol{F}Y)_u(t_0))\geq Y_u(t_0)-\hat{Y}_u(t_0)>0,$$
contradicting the complementarity property.
\end{proof}

We finish this subsection by showing that the solution $(Y,Z)$ to \eqref{eq:cond1op}--\eqref{eq:cond3op} will always be continuous in $t$ when the free process $X$ is continuous in $t$. 

\begin{lemma} \label{lem:contXcontY}
Suppose $X\in\mathcal{C}_T(L_1)$ and $F\in\mathcal{R}$. Then there exists a unique solution $(Y,Z)\in\mathcal{C}_T^\uparrow(L_1)\times\mathcal{C}_T(L_1)$ to the system defined by \eqref{eq:cond1op}--\eqref{eq:cond3op}.
\end{lemma}

\begin{proof}
From the above work, since $\mathcal{C}_T(L_1) \subset \mathcal{D}_T(L_1)$, it is guaranteed that there exists a unique solution $(Y,Z)\in\mathcal{D}_T^\uparrow(L_1)\times\mathcal{D}_T(L_1)$. Assume for a contradiction that $Y\notin \mathcal{C}_T^\uparrow(L_1)$. Then there exists some ${u^*}\in[0,1]$, $ t\in[0,T]$ such that $Y_{u^*}(t) - Y_{u^*}(t-) > 0$. By the complementarity property \eqref{eq:cond3op}, this then implies that $Z_{u^*}(t) = 0$. Since $X\in\mathcal{C}_T(L_1)$, it must be that $X_{u^*}(t)-X_{u^*}(t-)=0$. Putting these together in \eqref{eq:cond1op} results in
\begin{align*}
Y_{u^*}(t) - Y_{u^*}(t-)- ((\boldsymbol{F}Y)_{u^*}(t)- (\boldsymbol{F}Y)_{u^*}(t-)) = -Z_{u^*}(t-) \le 0.
\end{align*}
Define function $\Delta \in L_1([0,1])$ by $\Delta(u) = Y_u(t)-Y_u(t-)$ and note that $\Delta(u)>0$ whenever $Y_u(t)$ has a discontinuity at $t$, and $\Delta(u)=0$ otherwise. It then follows that
$$({\bf1}-\boldsymbol{F})\Delta \le 0.$$
Applying Lemma \ref{lem: invop}, it is clear that since $({\bf1}-\boldsymbol{F})^{-1}$ exists and is positive, 
$$\Delta \le ({\bf1}-\boldsymbol{F})^{-1}0 = 0,$$
thus contradicting that $Y_{u^*}(t)-Y_{u^*}(t-) > 0$. Therefore it must be that $Y\in \mathcal{C}_T^\uparrow(L_1)$, and it is also then immediate through \eqref{eq:cond1op} that $Z\in \mathcal{C}_T(L_1)$.
\end{proof}

\subsection{Continuity and Lipschitz Properties of the Infinite-Dimensional Skorokhod Map} \label{sec:lipschitz}
We have thus far shown existence and uniqueness of the infinite-dimensional Skorokhod map. 
We now show that the reflection and regulator maps satisfy Lipschitz properties in both $X$ and $\boldsymbol{F}$. 
This property is key in proving convergence in Section \ref{sec:convproof}.

We first present Lemma \ref{lem:pikcontraction} showing convergence of the mapping $\pi_{X,F}$ to the regulator $\Psi_F(X)$. 
The following elementary result will be needed, and the proof is omitted.

\begin{lemma}
\label{lem:sup-difference}
    For any two measurable functions $f,g \colon [0,T]\to\RR$ and $t \in [0,T]$, 
    $$\left|\sup_{0 \le s \le t} f(s) - \sup_{0 \le s \le t} g(s)\right| \le \sup_{0 \le s \le t} |f(s)-g(s)|.$$
\end{lemma}

\begin{lemma}
\label{lem:pikcontraction}
For $\boldsymbol{F}\in\mathcal{R}$ and any $W\in\mathcal{D}_T(L_1)$,$$\|\pi^n_{X,F}(W) - \Psi_{F}(X)\|_{T,1} \to 0 \quad\mbox{as $n\to\infty$}.$$
\end{lemma}

\begin{proof}
Take $W^1,W^2\in\mathcal{D}_T(L_1)$. From the definition of $\pi_{X,F}$, we have
\begin{align*}
&\sup_{0\leq t \leq T}|\pi_{X,F}(W^1)(t) - \pi_{X,F}(W^2)(t)| \\
=& \sup_{0\leq t \leq T}\bigg|\sup_{0\leq s \leq t}[-X(s) + \boldsymbol{F}W^1(s)]^+-\sup_{0\leq s \leq t}[-X(s) + \boldsymbol{F}W^2(s)]^+\bigg| \\
\le & \sup_{0\leq t \leq T}|(-X(t) + \boldsymbol{F}W^1(t))-(-X(t) + \boldsymbol{F}W^2(t))| \\
\le & \boldsymbol{F}\sup_{0\leq t \leq T}|W^1(t)-W^2(t)|,
\end{align*}
where we have used Lemma \ref{lem:sup-difference} in the first inequality.
Then since $\pi^k_{X,F}(W^1), \pi^k_{X,F}(W^2)\in\mathcal{D}_T(L_1)$ for any $k$,
\begin{equation}
\label{eq:pitopsi}
\begin{aligned}
\sup_{0\leq t \leq T}|\pi^n_{X,F}(W^1)(t) - \pi^n_{X,F}(W^2)(t)|  &\le \boldsymbol{F}\sup_{0\leq t \leq T}|\pi^{n-1}_{X,F}(W^1)(t) - \pi^{n-1}_{X,F}(W^2)(t)| \\
&\le \dots \le \boldsymbol{F}^{(n)}\sup_{0\leq t \leq T}|W^1(t)-W^2(t)|.
\end{aligned}
\end{equation}
Taking the $\|\cdot\|_{1}$ norm over \eqref{eq:pitopsi} then yields
$$\|\pi^n_{X,F}(W^1) - \pi^n_{X,F}(W^2)\|_{T,1} \le \|\boldsymbol{F}^{(n)}\|_{op}\|W^1 - W^2\|_{T,1}\,.$$
From Lemma \ref{lem: invop}, $\|\boldsymbol{F}^{(n)}\|_{op}\to0$ as $n\to\infty$ and so $\|\pi^n_{X,F}(W^1) - \pi^n_{X,F}(W^2)\|_{T,1}\to0$ as $n\to\infty$.

Hence, since $\Psi_{F}(X)$ is a fixed point of $\pi_{X,F}$, taking $W^2=\Psi_{F}(X)$ we have
\begin{equation*}
    \|\pi^n_{X,F}(W^1) - \Psi_{F}(X)\|_{T,1} \to 0 \quad\mbox{as } n\to\infty. \qedhere
\end{equation*}
\end{proof}

We now present two theorems stating Lipschitz continuity of the regulator and reflected process. In Theorem \ref{thm:lipschitz}, we show Lipschitz continuity in the free process $X$. Lemma \ref{lem:krefbound} shows Lipschitz continuity in the operator $\boldsymbol{F}$ and finally Theorem \ref{thm:kconvergence} combines these two, which can then be used for the proof of convergence.

\begin{remark}
For $\boldsymbol{F}\in\mathcal{R}$, given $\gamma$ and $k$ such that $\|\boldsymbol{F}^{(k)}\|_{op}\le\gamma$, we refer to $(\gamma,k)$ as \textbf{bounded parameters} of $\boldsymbol{F}$. Note that while bounded parameters of $\boldsymbol{F}$ always exist, they are not unique and in fact hold for any $\gamma\in(0,1)$ with corresponding $k$ depending on $\gamma$.
\end{remark}

\begin{theorem}
\label{thm:lipschitz}
For $\boldsymbol{F}\in\mathcal{R}$, the reflection map $R_F \coloneq (\Psi_F,\Phi_F)$ is Lipschitz.  Namely, for bounded parameters $(\gamma, k)$ we have
\begin{equation}
\label{eq:LipschitzReg}
\|\Psi_F(X^1)-\Psi_F(X^2)\|_{T,1}\leq\frac{k}{1-\gamma}\|X^1-X^2\|_{T,1}\,,
\end{equation}
\begin{equation}
\label{eq:LipschitzRef}
\|\Phi_F(X^1)-\Phi_F(X^2)\|_{T,1}\leq \left(1+\frac{2k}{1-\gamma}\right)\|X^1-X^2\|_{T,1}\,,
\end{equation}
for all $X^1,X^2 \in \mathcal{D}_T(L^1)$.
\end{theorem}

\begin{proof}
    Given any $X^1, X^2 \in \mathcal{D}_T(L_1)$, define
    $$\pi_i(Y)(t) \coloneq\pi_{X^i,F}(Y)(t) = \sup_{0\leq s \leq t}[-X^i(s) + \boldsymbol{F}Y(s)]^+, \quad i = 1,2, \quad t \in [0,T], \quad Y \in \mathcal{D}_T(L_1).$$
    For $X \in \mathcal{D}_T(L_1)$, define
    $$[\eta(X)]_u(t) := \sup_{0 \le s \le t} [X_u(s)]^+, \quad t \in [0,T], \quad u \in [0,1].$$
    We claim that (for each time and each coordinate)
    $$\Psi_F(X^1) \le \pi_2^n(\Psi(X^1)) + \sum_{i=0}^{n-1} \boldsymbol{F}^{(i)}\eta(X^2-X^1)$$
    for all $n = 0,1,\dotsc$.
    To see this, clearly it holds for $n=0$ with $\pi_2^0$ being the identity map.
    Now suppose the claim holds for $n$.
    Then we have
    \begin{align*}
        \Psi_F(X^1) & = \pi_1(\Psi_F(X^1)) = \eta(\boldsymbol{F}\Psi_F(X^1)-X^1) \\
        & \le \eta\left(\boldsymbol{F}\pi_2^n(\Psi_F(X^1)) - X^2 + \boldsymbol{F} \sum_{i=0}^{n-1} \boldsymbol{F}^{(i)}\eta(X^2-X^1) + (X^2-X^1)\right) \\
        & \le \eta\left(\boldsymbol{F}\pi_2^n(\Psi_F(X^1)) - X^2 + \sum_{i=0}^{n} \boldsymbol{F}^{(i)}\eta(X^2-X^1)\right) \\
        & \le \eta\left(\boldsymbol{F}\pi_2^{n}(\Psi_F(X^1)) - X^2\right) + \sum_{i=0}^{n} \boldsymbol{F}^{(i)}\eta(X^2-X^1) \\
        & = \pi_2^{n+1}(\Psi_F(X^1)) + \sum_{i=0}^{n} \boldsymbol{F}^{(i)}\eta(X^2-X^1).
    \end{align*}
    Therefore the claim holds for all $n$ by induction.
    Letting $n \to \infty$ and applying Lemma \ref{lem:pikcontraction}, from the claim we get for all $t\in[0,T]$ and a.e. $u\in[0,1]$,
    \begin{equation*}
        \big(\Psi_F(X^1) - \Psi_F(X^2)\big)_u(t) \le \big(\sum_{i=0}^{\infty} \boldsymbol{F}^{(i)}\eta(X^2-X^1)\big)_u(t).
    \end{equation*}
    By symmetry,
    \begin{equation*}
        \big(\Psi_F(X^1) - \Psi_F(X^2)\big)_u(t) \ge - \big(\sum_{i=0}^{\infty} \boldsymbol{F}^{(i)}\eta(X^1-X^2)\big)_u(t).
    \end{equation*}
    Therefore, using Lemma \ref{lem: invop} and applying $\|\cdot\|_{T,1}$, we obtain
    \begin{equation*}
        \|\Psi_F(X^1) - \Psi_F(X^2)\|_{T,1} \le \|(\boldsymbol{1}-\boldsymbol{F})^{-1}\|_{op} \|X^1-X^2\|_{T,1}.
    \end{equation*}
Hence,
\begin{equation*}
    \|\Psi_F(X^1) - \Psi_F(X^2)\|_{T,1} \le \frac{k}{1-\gamma} \|X^1-X^2\|_{T,1}.
\end{equation*}

For the reflected process, using \eqref{eq:cond1op}  and noting that $\|\boldsymbol{1}-\boldsymbol{F}\|_{op} \leq 2$ yields
\begin{align*}
    \|\Phi_F(X^1) - \Phi_F(X^2)\|_{T,1} &\leq \|X^1-X^2\|_{T,1}+\|\boldsymbol{1}-\boldsymbol{F}\|_{op}\cdot\|\Psi_F(X^1) - \Psi_F(X^2)\|_{T,1} \\
    &\leq (1 + 2\cdot\|(\boldsymbol{1}-\boldsymbol{F})^{-1}\|_{op})\|X^1 - X^2\|_{T,1} \,.
\end{align*}
From this \eqref{eq:LipschitzRef} is immediate.
\end{proof}

In Theorem \ref{thm:lipschitz}, we showed that the reflection map is Lipschitz continuous in the free process $X$. We now extend this Lipschitz continuity to include the reflection operator $\boldsymbol{F}$. First, we prove some necessary results.

\begin{lemma}
\label{lem:krefbound}
Let $\boldsymbol{F}_1$ and $\boldsymbol{F}_2$ be two operators in $\mathcal{R}$ with respective bounded parameters $(\gamma_1,k_1)$ and $(\gamma_2,k_2)$. Then for any $X\in \mathcal{D}_T(L_1)$ and all $n\geq1$, $m\geq0$, $i=1,2$,
\begin{enumerate}[label=(\roman*)]

\item
\begin{equation}
\label{eq:kpibound0}
\|\boldsymbol{F}_i^{(m)}\pi^n_{X,F_i}(0)\|_{T,1} \leq \|\boldsymbol{F}_i^{(m)}\|_{op}\|X\|_{T,1} + \|\boldsymbol{F}_i^{(m+1)}\pi^{n-1}_{X,F_i}(0)\|_{T,1}\,,
\end{equation}

\item
\begin{equation}
\label{eq:kpibound1}
\|\pi^n_{X,F_i}(0)\|_{T,1} \leq \bigg(\sum_{j=0}^{n-1}\|\boldsymbol{F}_i^{(j)}\|_{op}\bigg)\|X\|_{T,1}\,,
\end{equation}

\item
\begin{equation}
\label{eq:kpibound1.5}
\begin{split}
&\boldsymbol{F}_2^{(m)}\sup_{0\leq s\leq t}\big|(\pi^n_{X,F_1}(0))(s) - (\pi^n_{X,F_2}(0))(s)\big| \\&\leq \boldsymbol{F}_2^{(m)}\sup_{0\le s\le t}\big|(\boldsymbol{F}_1-\boldsymbol{F}_2)(\pi^{n-1}_{X,F_1}(0))(s)\big|  \\
& \quad+ \boldsymbol{F}_2^{(m+1)}\sup_{0\leq s\leq t}\big|(\pi^{n-1}_{X,F_1}(0))(s) - (\pi^{n-1}_{X,F_2}(0))(s)\big|\,, 
\end{split}
\end{equation}

\item
\begin{equation}
\label{eq:kpibound2}
\|\pi^n_{X,F_1}(0) - \pi^n_{X,F_2}(0)\|_{T,1} \leq \bigg(\sum_{j=0}^{n-1}\|\boldsymbol{F}_2^{(j)}\|_{op}\bigg)\frac{k_1\|\boldsymbol{F}_1-\boldsymbol{F}_2\|_{op}\|X\|_{T,1}}{1-\gamma_1}\,,
\end{equation}

\item
\begin{equation}
\label{eq:kregbound1}
\|\Psi_{F_i}(X)\|_{T,1} \leq \frac{k_i\|X\|_{T,1}}{1-\gamma_i}\,,
\end{equation}

\item
\begin{equation}
\label{eq:kregbound2}
\|\Psi_{F_1}(X)-\Psi_{F_2}(X)\|_{T,1} \leq \frac{k_1k_2\|\boldsymbol{F}_1-\boldsymbol{F}_2\|_{op}\|X\|_{T,1}}{(1-\gamma_1)(1-\gamma_2)}\,. 
\end{equation}
\end{enumerate}
\end{lemma}

\begin{proof}
\begin{enumerate}[label=(\roman*)]

\item
Using the fact that $\sup_{0\le t\le T}(\boldsymbol{F}W(t))\leq \boldsymbol{F}(\sup_{0\le t\le T} W(t))$ and the observation that $\pi_{X,F_1}$ is increasing in $t$, 
\begin{align*}
\|\boldsymbol{F}_i^{(m)}\pi^n_{X,F_i}(0)\|_{T,1} &= \bigg\|\sup_{0\leq t\leq T} \bigg|\boldsymbol{F}_i^{(m)}\sup_{0\leq s\leq t}\big[-X(s)+(\boldsymbol{F}_i\pi^{n-1}_{X,F_i}(0))(s)\big]^+\bigg| \bigg\|_1 \\
&\leq \bigg\|\sup_{0\leq t\leq T}\bigg|\boldsymbol{F}_i^{(m)}\sup_{0\leq s\leq t}|X(s)|\bigg|+\sup_{0\leq t\leq T}\bigg|(\boldsymbol{F}_i^{(m+1)}\sup_{0\leq s\leq t}\pi^{n-1}_{X,F_i}(0)(s))\bigg| \bigg\|_1 \\
&\leq \|\boldsymbol{F}_i^{(m)}\|_{op} \bigg\|\sup_{0\leq t\leq T}|X(t)|\bigg\|_1 + \bigg \|\sup_{0\leq t\leq T}(\boldsymbol{F}_i^{(m+1)}\pi^{n-1}_{X,F_i}(0))(t)\bigg\|_1 \\
&= \|\boldsymbol{F}_i^{(m)}\|_{op}\|X\|_{T,1} + \|\boldsymbol{F}_i^{(m+1)}\pi^{n-1}_{X,F_i}(0)\|_{T,1}\,.
\end{align*}

\item
Iteratively using \eqref{eq:kpibound0}, we get
\begin{align*}
\|\pi^n_{X,F_i}(0)\|_{T,1} &\leq \|X\|_{T,1} + \|\boldsymbol{F}_i\pi^{n-1}_{X,F_i}(0)\|_{T,1} \\
&\leq \|X\|_{T,1} + \|\boldsymbol{F}_i\|_{op}\|X\|_{T,1} + \|\boldsymbol{F}_i^{(2)}\pi^{n-2}_{X,F_i}(0)\|_{T,1} \\
&\leq \|X\|_{T,1} + \|\boldsymbol{F}_i\|_{op}\|X\|_{T,1} + \|\boldsymbol{F}_i^{(2)}\|_{op}\|X\|_{T,1} + \|\boldsymbol{F}_i^{(3)}\pi^{n-3}_{X,F_i}(0)\|_{T,1} \\
&\leq\cdots\leq \bigg(\sum_{j=0}^{n-2}\|\boldsymbol{F}_i^{(j)}\|_{op}\bigg)\|X\|_{T,1} + \|\boldsymbol{F}_i^{(n-1)}\pi_{X,F_i}(0)\|_{T,1}\,. 
\end{align*}
Since $\|\pi_{X,F_i}(0)\|_{T,1} = \|\sup_{0\leq t\leq T}[-X(t)]^+\|_1 \leq \|\sup_{0\leq t\leq T}|X(t)|\|_1 = \|X\|_{T,1}$, the result follows.

\item Using Lemma \ref{lem:sup-difference} in the third line, 

\begin{align*}
&\boldsymbol{F}_2^{(m)}\sup_{0\leq s\leq t}\big|(\pi^n_{X,F_1}(0))(s) - (\pi^n_{X,F_2}(0))(s) \big| \\
&= \boldsymbol{F}_2^{(m)}\sup_{0\leq s\leq t}\bigg|\sup_{0\leq r\leq s}\big[-X(r)+(\boldsymbol{F}_1\pi^{n-1}_{X,F_1}(0))(r)\big]^+ - \sup_{0\leq r\leq s}\big[-X(r)+(\boldsymbol{F}_2\pi^{n-1}_{X,F_2}(0))(r)\big]^+\bigg| \\
&\leq \boldsymbol{F}_2^{(m)}\sup_{0\leq s\leq t}\bigg|\big[-X(s)+(\boldsymbol{F}_1\pi^{n-1}_{X,F_1}(0))(s)\big]^+ - \big[-X(s)+(\boldsymbol{F}_2\pi^{n-1}_{X,F_2}(0))(s)\big]^+\bigg| \\
&\leq \boldsymbol{F}_2^{(m)}\sup_{0\leq s\leq t} \big|(\boldsymbol{F}_1\pi^{n-1}_{X,F_1}(0))(s) - (\boldsymbol{F}_2\pi^{n-1}_{X,F_2}(0))(s)\big| \\
&\leq \boldsymbol{F}_2^{(m)}\sup_{0\leq s\leq t}\big|(\boldsymbol{F}_1\pi^{n-1}_{X,F_1}(0))(s) - (\boldsymbol{F}_2\pi^{n-1}_{X,F_1}(0))(s)\big| \\
&\qquad+ \boldsymbol{F}_2^{(m)}\sup_{0\leq s\leq t}\big|(\boldsymbol{F}_2\pi^{n-1}_{X,F_1}(0))(s) - (\boldsymbol{F}_2\pi^{n-1}_{X,F_2}(0))(s)\big|\\
&\leq \boldsymbol{F}_2^{(m)}\sup_{0\le s\le t}\big|(\boldsymbol{F}_1-\boldsymbol{F}_2)(\pi^{n-1}_{X,F_1}(0))(s)\big| + \boldsymbol{F}_2^{(m+1)}\sup_{0\leq s\leq t}\big|(\pi^{n-1}_{X,F_1}(0))(s) - (\pi^{n-1}_{X,F_2}(0))(s))\big|. 
\end{align*}

\item
Using \eqref{eq:kpibound1.5},
\begin{align*}
&\sup_{0\leq s\leq t}\big|(\pi^n_{X,F_1}(0))(s) - (\pi^n_{X,F_2}(0))(s)\big| \\
&\leq \sup_{0\le s\le t} \big|(\boldsymbol{F}_1-\boldsymbol{F}_2)(\pi^{n-1}_{X,F_1}(0))(s)\big| + \boldsymbol{F}_2\sup_{0\leq s\leq t}\big|(\pi^{n-1}_{X,F_1}(0))(s) - \pi^{n-1}_{X,F_2}(0))(s)\big| \\
&\leq \sup_{0\le s\le t} \big|(\boldsymbol{F}_1-\boldsymbol{F}_2)(\pi^{n-1}_{X,F_1}(0))(s)\big| + \boldsymbol{F}_2 \sup_{0\le s\le t} \big|(\boldsymbol{F}_1-\boldsymbol{F}_2)(\pi^{n-2}_{X,F_1}(0))(s)\big| \\
&\qquad+ \boldsymbol{F}_2^{(2)}\sup_{0\leq s\leq t}\big|(\pi^{n-2}_{X,F_1}(0))(s) - \pi^{n-2}_{X,F_2}(0))(s)\big| \\
&\leq \cdots \leq \sum_{j=0}^{n-1}F_2^{(j)} \sup_{0\le s\le t} |(\boldsymbol{F}_1-\boldsymbol{F}_2)(\pi^{n-j-1}_{X,F_1}(0))(s)| + \boldsymbol{F}_2^{(n)}\cdot 0\,.
\end{align*}
Taking the norm $\|\cdot\|_{T,1}$ over the inequality and using \eqref{eq:kpibound1} and Lemma \ref{lem: invop} yields

\begin{align*}
\|\pi^n_{X,F_1}(0) - \pi^n_{X,F_2}(0)\|_{T,1} &= \bigg\|\sup_{0\leq t\leq T}\sup_{0\leq s\leq t}\big|(\pi^n_{X,F_1}(0))(s) - (\pi^n_{X,F_2}(0))(s)\big|\bigg\|_1 \\
&\leq \bigg\|\sup_{0\leq t\leq T}\sum_{j=0}^{n-1}\boldsymbol{F}_2^{(j)} \sup_{0\le s\le t} |(\boldsymbol{F}_1-\boldsymbol{F}_2)(\pi^{n-j-1}_{X,F_1}(0))(s)|\bigg\|_1 \\
&\leq \sum_{j=0}^{n-1}\|\boldsymbol{F}_2^{(j)}\|_{op}\|\boldsymbol{F}_1-\boldsymbol{F}_2\|_{op}\left(\sum_{l=0}^{n-j}\|\boldsymbol{F}_1^{(l)}\|_{op}\right)\|X\|_{T,1} \\
&\leq \sum_{j=0}^{n-1}\|\boldsymbol{F}_2^{(j)}\|_{op}\|\boldsymbol{F}_1-\boldsymbol{F}_2\|_{op}\left(\sum_{l=0}^\infty\|\boldsymbol{F}_1^{(l)}\|_{op}\right)\|X\|_{T,1} \\
&\leq \left(\sum_{j=0}^{n-1}\|\boldsymbol{F}_2^{(j)}\|_{op}\right)\frac{k_1\|\boldsymbol{F}_1-\boldsymbol{F}_2\|_{op}\|X\|_{T,1}}{1-\gamma_1}
\end{align*}

\item
From \eqref{eq:kpibound1},
\begin{align*}
\|\Psi_{F_i}(X)\|_{T,1} &\leq \|\Psi_{F_i}(X) - \pi^n_{X,F_i}(0)\|_{T,1} + \|\pi^n_{X,F_i}(0)\|_{T,1} \\
&\leq \|\Psi_{F_i}(X) - \pi^n_{X,F_i}(0)\|_{T,1} + \left(\sum_{j=0}^{n-1}\|\boldsymbol{F}_i^{(j)}\|_{op}\right)\|X\|_{T,1}
\end{align*}
Finally, applying Lemma \ref{lem:pikcontraction}  with $\Psi_{F_i}(X) = \pi_{X,F_i}(\Psi_{F_i}(X))$,
letting $n\to\infty$ yields the result.

\item
From \eqref{eq:kpibound2},
\begin{align*}
\|\Psi_{F_1}(X) - \Psi_{F_2}(X)\|_{T,1} &\leq \|\Psi_{F_1}(X) - \pi^n_{X,F_1}(0)\|_{T,1} + \|\pi^n_{X,F_1}(0)-\pi^n_{X,F_2}(0)\|_{T,1} \\&+ \|\pi^n_{X,F_2}(0) - \Psi_{F_2}(X)\|_{T,1} \\
&\leq \|\Psi_{F_1}(X) - \pi^n_{X,F_1}(0)\|_{T,1} + \left(\sum_{j=0}^{n-1}\|\boldsymbol{F}_2^{(j)}\|_{op}\right)\frac{k_1\|\boldsymbol{F}_1-\boldsymbol{F}_2\|_{op}\|X\|_{T,1}}{1-\gamma_1} \\&+ \|\pi^n_{X,F_2}(0) - \Psi_{F_2}\|_{T,1}\,. 
\end{align*}
Again letting $n\to\infty$ yields the result. \qedhere
\end{enumerate}
\end{proof}

\begin{theorem}
\label{thm:kconvergence}
Take $X^1\in\mathcal{D}_T(L_1)$ and $X^2\in\mathcal{D}_T(L_1)$. Let $(\Phi_{F_1}(X^1),\Psi_{F_1}(X^1))$ and $(\Phi_{F_2}(X^2),\Psi_{F_2}(X^2))$ be the corresponding reflected processes and regulators associated with operators $\boldsymbol{F}_1\in\mathcal{R}$ and $\boldsymbol{F}_2\in\mathcal{R}$ with respective bounded parameters $(\gamma_1,k_1)$ and $(\gamma_2,k_2)$.
Then the following holds:
\begin{enumerate}[label=(\roman*)]
\item
\begin{equation}
\|\Psi_{F_2}(X^2)-\Psi_{F_1}(X^1)\|_{T,1} \leq \frac{k_2}{1-\gamma_2}\|X^2-X^1\|_{T,1} + \frac{k_1k_2\|X^1\|_{T,1}}{(1-\gamma_1)(1-\gamma_2)}\|\boldsymbol{F}_2-\boldsymbol{F}_1\|_{op}\,,
\end{equation}
\item
\begin{equation}
\begin{split}
    \|\Phi_{F_2}(X^2)-\Phi_{F_1}(X^1)\|_{T,1}  &\leq \left(1+\frac{2k_2}{1-\gamma_2}\right)\|X^2-X^1\|_{T,1} \\ & \quad + \left(\frac{2k_1k_2}{(1-\gamma_1)(1-\gamma_2)}+\frac{k_1}{1-\gamma_1}\right)\|X^1\|_{T,1}\|\boldsymbol{F}_2-\boldsymbol{F}_1\|_{op}\,.
\end{split}
\end{equation}
\end{enumerate}
\end{theorem}

\begin{proof}
Using \eqref{eq:kregbound2} from Lemma \ref{lem:krefbound} and Theorem \ref{thm:lipschitz}, we obtain
\begin{align*}
\|\Psi_{F_2}(X^2)-\Psi_{F_1}(X^1)\|_{T,1} &\leq \|\Psi_{F_2}(X^2)-\Psi_{F_2}(X^1)\|_{T,1} + \|\Psi_{F_2}(X^1)-\Psi_{F_1}(X^1)\|_{T,1} \\
&\leq \frac{k_2}{1-\gamma_2}\|X^2-X^1\|_{T,1} + \frac{k_1k_2\|X^1\|_{T,1}}{(1-\gamma_1)(1-\gamma_2)}\|\boldsymbol{F}_2-\boldsymbol{F}_1\|_{op}\,.
\end{align*}
Then, using the relations $\Phi_{F_1}(X^1) = X^1 + (\mathbf{1}-\boldsymbol{F}_1)\Psi_{F_1}(X^1)$ and $\Phi_{F_2}(X^2) = X^2 + (\mathbf{1}-\boldsymbol{F}_2)\Psi_{F_2}(X^2)$, we have
$$\Phi_{F_2}(X^2)-\Phi_{F_1}(X^1) = (X^2-X^1)+(\mathbf{1}-\boldsymbol{F}_2)(\Psi_{F_2}(X^2)-\Psi_{F_1}(X^1))+(\boldsymbol{F}_2-\boldsymbol{F}_1)\Psi_{F_1}(X^1)\,.$$
Taking the norm and using \eqref{eq:kregbound1} gives
\begin{align*}
&\|\Phi_{F_2}(X^2)-\Phi_{F_1}(X^1)\|_{T,1} 
\\&\leq \|X^2-X^1\|_{T,1}+\|\mathbf{1}-\boldsymbol{F}_2\|_{op}\|\Psi_{F_2}(X^2)-\Psi_{F_1}(X^1)\|_{T,1}+\|\boldsymbol{F}_2-\boldsymbol{F}_1\|_{op}\|\Psi_{F_1}(X^1)\|_{T,1} \\
&\leq \|X^2-X^1\|_{T,1} + 2\left(\frac{k_2}{1-\gamma_2}\|X^2-X^1\|_{T,1} + \frac{k_1k_2\|X^1\|_{T,1}}{(1-\gamma_1)(1-\gamma_2)}\|\boldsymbol{F}_2-\boldsymbol{F}_1\|_{op}\right) \\&\quad+\|\boldsymbol{F}_2-\boldsymbol{F}_1\|_{op}\frac{k_1\|X^1\|_{T,1}}{1-\gamma_1}\\
&= \left(1+\frac{2k_2}{1-\gamma_2}\right)\|X^2-X^1\|_{T,1} + \left(\frac{2k_1k_2}{(1-\gamma_1)(1-\gamma_2)}+\frac{k_1}{1-\gamma_1}\right)\|X^1\|_{T,1}\|\boldsymbol{F}_2-\boldsymbol{F}_1\|_{op}\,. 
\end{align*}
\qedhere
\end{proof}

\begin{lemma}
\label{lem:opseqnorm}
Take $\boldsymbol{F} \in \mathcal{R}$ and let $\{\boldsymbol{F}_n\}_{n\ge1}$ be a sequence of operators in $\mathcal{R}$  such that for some constants $C>0$ and $\alpha>0$, $\|\boldsymbol{F}_n - \boldsymbol{F}\|_{op}\to0$ as $n\to\infty$. Then for any $\gamma\in(0,1)$ independent of $n$, there exists $k\ge 1$ such that $\|\boldsymbol{F}^{(k)}\|_{op}\le\gamma$ and for sufficiently large $n$, $\|\boldsymbol{F}_n^{(k)}\|_{op}\le\gamma$.
\end{lemma}

\begin{proof}
We first use the identity
$$\boldsymbol{F}^{(k)}_n - \boldsymbol{F}^{(k)} = \sum_{j=0}^{k-1}\boldsymbol{F}^{(k-1-j)}_n(\boldsymbol{F}_n - \boldsymbol{F})\boldsymbol{F}^{(j)},$$
and thus
\begin{equation}
\label{eq:opseqpow}
\|\boldsymbol{F}^{(k)}_n - \boldsymbol{F}^{(k)}\|_{op} \le \|(\boldsymbol{F}_n - \boldsymbol{F})\|_{op}\sum_{j=0}^{k-1}\|\boldsymbol{F}^{(k-1-j)}_n\|_{op}\|\boldsymbol{F}^{(j)}\|_{op} \le k\|(\boldsymbol{F}_n - \boldsymbol{F})\|_{op}\,.
\end{equation}
From Lemma \ref{lem: invop} $\exists k\ge1$ such that 
$\|\boldsymbol{F}^{(k)}\|\le\frac{\gamma}{2}$. Using \eqref{eq:opseqpow}, given such a $k$ we can find an $N\ge1$ such that $\forall n\ge N$, $\|\boldsymbol{F}^{(k)}_n - \boldsymbol{F}^{(k)}\|_{op} \le \gamma/2$. Putting these together yields
\begin{align*}
\|\boldsymbol{F}^{(k)}_n\|_{op} \le \|\boldsymbol{F}^{(k)}\|_{op} + \|\boldsymbol{F}^{(k)}_n - \boldsymbol{F}^{(k)}\|_{op} \le \gamma/2 + \gamma/2 = \gamma
\end{align*}
as required.
\end{proof}

\begin{theorem}
\label{thm:kconvergenceseq}
Take $X\in\mathcal{D}_T(L_1)$ and let $\{X^n\}_{n\geq1}$ be a sequence in $\mathcal{D}_T(L_1)$. 
Let $(\Phi_F(X),\Psi_F(X))$ and $\left(\{\Phi_{F_n}(X^n)\}_{n\geq1},\{\Psi_{F_n}(X^n)\}_{n\geq1}\right)$ be the corresponding reflected processes and regulators associated with operators $\boldsymbol{F}\in\mathcal{R}$ and $\{\boldsymbol{F}_n\}_{n\geq1}\in\mathcal{R}$ satisfying $\|\boldsymbol{F}_n - \boldsymbol{F}\|_{op}\to0$ as $n\to\infty$.
Then given any $\gamma\in(0,1)$ there exists $k\ge1$ such that for sufficiently large $n$ the following hold:
\begin{enumerate}[label=(\roman*)]
\item
\begin{equation}
\|\Psi_{F_n}(X^n)-\Psi_F(X)\|_{T,1} \leq \frac{k}{1-\gamma}\|X^n-X\|_{T,1} + \frac{k^2\|X\|_{T,1}}{(1-\gamma)^2}\|\boldsymbol{F}_n-\boldsymbol{F}\|_{op}\,,
\end{equation}
\item
\begin{equation}
\begin{split}
    \|\Phi_{F_n}(X^n)-\Phi_F(X)\|_{T,1}  &\leq \left(1+\frac{2k}{1-\gamma}\right)\|X^n-X\|_{T,1} \\ & \quad + \left(\frac{2k^2}{(1-\gamma)^2}+\frac{k}{1-\gamma}\right)\|X\|_{T,1}\|\boldsymbol{F}_n-\boldsymbol{F}\|_{op}\,.
\end{split}
\end{equation}
\end{enumerate}
\end{theorem}

\begin{proof}
The results follow from applying Lemma \ref{lem:opseqnorm} to Theorem \ref{thm:kconvergence}.
\end{proof}

\allowdisplaybreaks

\medskip

\section{Proofs for Convergence Results} 
\label{sec:convproof}

\subsection{Approximation of Queueing Process with Intermediate Process}
First note that we can write the queueing process $Q_i^N(t)$ as
\begin{align*} 
Q_i^N(t) &= Q_i^N(0) + A_i^N(t) + \sum_{j=1}^N   p^N_{ji} \mu_j^N B_j^N(t) -\sum_{j=0}^N p^N_{ij} \mu_i^N B_i^N(t) \\
& \quad + \sum_{j=1}^N  (\Poisson^N_{ji}(\mu_j^N B_j^N(t)) - p^N_{ji} \mu_j^N B_j^N(t)) - \sum_{j=0}^N (\Poisson^N_{ij}(\mu_i^N B_i^N(t)) - p^N_{ij} \mu_i^N B_i^N(t))  \\
&= Q_i^N(0) + \lambda^N_it + \Big(\frac{1}{N}\sum_{j=1}^N   G^N_{ji} \mu_j^N - \mu_i^N\Big) t +( A_i^N(t)-\lambda^N_it )  \\
& \quad+ \sum_{j=1}^N  (\Poisson^N_{ji}(\mu_j^N B_j^N(t)) - \frac{1}{N}G^N_{ji} \mu_j^N B_j^N(t)) - \sum_{j=0}^N (\Poisson^N_{ij}(\mu_i^N B_i^N(t)) - \frac{1}{N}G^N_{ij} \mu_i^N B_i^N(t))  \\
&  \quad-  \frac{1}{N}\sum_{j=1}^N   G^N_{ji} \mu_j^N I_j^N(t) + \mu_i^N I_i^N(t).
\end{align*}
Using \eqref{eq: scaledprocess}, the scaled process $\bar{Q}_i^N(t)$ can then be written as follows: 
\begin{equation}  \label{eqn-barQNi-rep1}
\begin{split}
\bar{Q}_i^N(t)
&= \bar{Q}_i^N(0) + \lambda^N_i t  +\Big(\frac{1}{N}\sum_{j=1}^N   G^N_{ji} \mu_j^N - \mu_i^N\Big) t  + ( \bar{A}_i^N(t)-\lambda^N_it )  \\
& \quad + \frac{1}{N^\alpha}\sum_{j=1}^N  \Big(\Poisson^N_{ji}(\mu_j^N N^\alpha\bar{B}_j^N(t)) - p^N_{ji} \mu_j^N N^\alpha\bar{B}_j^N(t) \Big)  \\
& \quad -\frac{1}{N^\alpha} \sum_{j=0}^N \Big(\Poisson^N_{ij}(\mu_i^N N^\alpha\bar{B}_i^N(t)) - p^N_{ij} \mu_i^N N^\alpha\bar{B}_i^N(t) \Big)  \\
& \quad -  \frac{1}{N}\sum_{j=1}^N   G^N_{ji} \mu_j^N \bar{I}_j^N(t) +\mu_i^N \bar{I}_i^N(t)\,.
\end{split}
\end{equation}
Note that $\bar{I}^N_i(t)$ satisfies 
\begin{equation}
\int_0^\infty 1_{\{\bar{Q}^N_i(t)>0\}} d\bar{I}^N_i(t)=0, \quad \forall\,i=1,\dots,N.
\end{equation}
In matrix form, we have
\begin{equation}\label{eqn-barQN-rep1}
\bar{Q}^N = \bar{X}^N + \big(I-\frac{1}{N}(G^N)^T \big) \bar{Y}^N,
\end{equation}
where 
\begin{equation} \label{eqn-barXN-rep}
\begin{split}
\bar{X}^N_i(t) &=\bar{Q}_i^N(0) + \lambda^N_it  +\Big(\frac{1}{N}\sum_{j=1}^N   G^N_{ji} \mu_j^N - \mu_i^N\Big) t + (\bar{A}^N_i(t) - \lambda^N_it)\\
& \quad + \frac{1}{N^\alpha}\sum_{j=1}^N  \Big(\Poisson^N_{ji}(\mu_j^N N^\alpha\bar{B}_j^N(t)) - p^N_{ji} \mu_j^N N^\alpha\bar{B}_j^N(t) \Big) \\
& \quad -\frac{1}{N^\alpha} \sum_{j=0}^N \Big(\Poisson^N_{ij}(\mu_i^N N^\alpha\bar{B}_i^N(t)) - p^N_{ij} \mu_i^N N^\alpha\bar{B}_i^N(t) \Big)  \,, 
\end{split}
\end{equation}
and
\begin{align} \label{eqn-barYNi}
\bar{Y}^N_i = \mu^N_i\bar{I}^N_i\,.
\end{align}

{\bf The intermediate process:}
To bridge the gap between the original stochastic process and the deterministic limiting process, we introduce an intermediate process $\breve{Q}^N_{i}(t)$, which corresponds to a network model with the same size as the original network, and is constructed from $\bar{Q}^N_{i}(t)$ by removing the compensated Poisson processes. Writing $$\breve{X}^N_{i}(t) = \bar{Q}^N_i(0) + \lambda^N_it + \Big(\frac{1}{N}\sum_{j=1}^N   G^N_{ji} \mu_j^N - {\mu_i^N}\Big) t, \quad i=1,\dots,N,$$
and noting that $\breve{X}^N\in\mathcal{C}_T^N$, by applying \cite[Theorem 1]{HarrisonReiman1981reflected} we see that there is a unique solution $(\breve{Q}^N,\breve{I}^N)\in\mathcal{C}_T^N\times\mathcal{C}_T^{\uparrow N}$ (where $\mathcal{C}_T^{\uparrow N}$ is the subspace of increasing functions in $\mathcal{C}_T^N$) satisfying
\begin{align} \label{eqn-breveQNi}
 \breve{Q}^N_{i}(t) &= \bar{Q}^N_i(0) + \lambda^N_it + \Big(\frac{1}{N}\sum_{j=1}^N   G^N_{ji} \mu_j^N - {\mu_i^N}\Big) t -  \frac{1}{N}\sum_{j=1}^N   G^N_{ji} \mu_j^N \breve{I}_j^N(t) +\mu_i^N \breve{I}_i^N(t)
 \end{align}
 where $\breve{Q}^N\ge0$, $\breve{I}_i^N(t)=0$, $d\breve{I}_i^N(t)\ge0$ and $\breve{I}_i^N(t)$ satisfies 
 \begin{equation}
\label{eq:breveINi}
\int_0^\infty 1_{\{\breve{Q}^N_i(t)>0\}} d\breve{I}^N_i(t)=0, \quad \forall\,i=1,\dots,N.
\end{equation}

 In vector form, we write 
\[\breve{Q}^N = \breve{X}^N + (I-\frac{1}{N}(G^N)^T) \breve{Y}^N,\] where \[\breve{X}^N(t) = \bar{Q}^N(0) + \lambda^N t - (I-\frac{1}{N}(G^N)^T)\mu^Nt, \quad \breve{Y}^N = \operatorname{diag}(\mu^N)\breve{I}^N.\]

As this is just a finite-dimensional Skorokhod problem, it is clear that $\breve{Q}$ and $\breve{Y}$ are the reflected process and regulator of $\breve{X}$,  respectively.

\begin{lemma}
\label{lem:compqueue}
Under Assumption \ref{assumption:GNGop}, there exists some $K \in (0,\infty)$ such that
\begin{equation}
\label{eq:compqueue}
\frac{1}{N} \sum_{i=1}^N \E \sup_{0 \le t \le T} |\bar{Q}_i^N(t)-\breve{Q}_i^N(t)|\le K \frac{1}{N} \sum_{i=1}^N \E \sup_{0 \le t \le T} |\bar{X}_i^N(t)-\breve{X}_i^N(t)|.
\end{equation}
\end{lemma}

\begin{proof}
This is a discrete version of the Lipschitz property seen in Theorem \ref{thm:lipschitz}.
See \cite{Whitt2002} for more detail. One should note that in \cite{Whitt2002}, $K$ depends on the size of the network through $\frac{N}{1-\gamma}$, where $\gamma\in(0,1)$ is such that $(P^N)^N$ has all row sums less than $\gamma$. While this dependence on $N$ is problematic as $N\to\infty$, Assumption \ref{assumption:GNGop} along with Lemma \ref{lem: invop} guarantees a tighter row sum bound of $\frac{k}{1-\gamma}$ uniformly as $N\to\infty$. A proof of this, relating matrix multiplication with operator composition, is given in Appendix \ref{sec:additional-pf}.
\end{proof}

We now find a bound on the right hand side of \eqref{eq:compqueue}.
First note that
\begin{align*}
\bar{X}^N_i(t)-\breve{X}^N_i(t) & = \bar{A}^N_i(t) - \lambda^N_it \\
& \quad+\frac{1}{N^\alpha}\sum_{j=1}^N \Big (\Poisson^N_{ji}(\mu_j^N N^\alpha\bar{B}_j^N(t)) - p^N_{ji} \mu_j^N N^\alpha\bar{B}_j^N(t) \Big) \\
&\quad -\frac{1}{N^\alpha} \sum_{j=0}^N \Big(\Poisson^N_{ij}(\mu_i^N N^\alpha\bar{B}_i^N(t)) - p^N_{ij} \mu_i^N N^\alpha\bar{B}_i^N(t)\Big)\,.
\end{align*}
Also note that the $j=i$ terms in the second and third summand cancel each other out.

The following is a standard result on the martingale property of Poisson processes.

\begin{lemma}
\label{lem:randomtimechange}
Let $N(t)$ be a Poisson process of rate 1. Let $\{Z(t)\}_{t\ge0}$ be a continuous increasing process  with $Z(0)=0$ and set $M(t)=N(Z(t))-Z(t)$. Then $M(t)$ is a martingale with respect to its natural filtration. Furthermore, the quadratic variation of $M$ is given by $\langle M \rangle_t = Z(t)$.
\end{lemma}

\begin{proof}
First note that since $Z$ is continuous, it is a predictable process. The result follows from the random time change theorem for Poisson processes (see, e.g., Theorem 6.1.3 in \cite{ethier2009markov}). 
\end{proof}

The following lemma provides estimates on terms in \eqref{eq:compqueue}.

\begin{lemma}
\label{lem:compbound}
We have 
\begin{equation*}
\frac{1}{N} \sum_{i=1}^N \E \sup_{0 \le t \le T} \left|\bar{X}^N_i(t) -\breve{X}^N_i(t)\right| \le \frac{2T^\frac{1}{2}}{N^{\alpha/2}}\sqrt{\|\lambda^N\|_1 + 2\|\mu^N\|_1} \,\,.
\end{equation*}
Hence, 
\begin{equation*}
\frac{1}{N} \sum_{i=1}^N \E \sup_{0 \le t \le T} \left|\bar{Q}^N_i(t) -\breve{Q}^N_i(t)\right| \le \frac{CT^\frac{1}{2}}{N^{\alpha/2}}\sqrt{\|\lambda^N\|_1 + 2\|\mu^N\|_1}
\end{equation*}for some constant $C>0$.
\end{lemma}

\begin{proof}
For fixed $i \in [N]$, let 
\begin{align*}
&M^N_i(t)= A^N_i(N^\alpha t) - \lambda^N_iN^\alpha t ,\\
&M^N_{ji}(t)=\Poisson^N_{ji}(\mu_j^N N^\alpha\bar{B}_j^N(t)) - p^N_{ji} \mu_j^N N^\alpha\bar{B}_j^N(t), \quad j\in[N]\backslash \{i\}\\
&M^N_{ij}(t)=\Poisson^N_{ij}(\mu_i^N N^\alpha\bar{B}_i^N(t)) - p^N_{ij} \mu_i^N N^\alpha\bar{B}_i^N(t), \quad j\in\{0\}\cup[N]\backslash \{i\}. 
\end{align*}
Note that we can write 
\begin{equation}
\label{eq:intermediate}
\bar{X}^N_i(t)-\breve{X}^N_i(t) = \frac{1}{N^\alpha}\left( M^N_i(t) + \sum_{j\in[N]\backslash \{i\}}M^N_{ji}(t) {}-{} \sum_{j\in\{0\}\cup[N]\backslash \{i\}}M^N_{ij}(t)\right).
\end{equation}
Applying Lemma \ref{lem:randomtimechange} yields
\begin{align*}
&\langle M^N_i\rangle_t =\lambda^N_iN^\alpha t, \\
&\langle M^N_{ji}\rangle_t =p^N_{ji} \mu_j^N N^\alpha\bar{B}_j^N(t), \\
&\langle M^N_{ij}\rangle_t=p^N_{ij} \mu_i^N N^\alpha\bar{B}_i^N(t)\,.
\end{align*}
Since each martingale term has no cross variation, it follows that
\begin{align}
\langle \bar{X}^N_i -\breve{X}^N_i \rangle_t &= \frac{1}{N^{2\alpha}}\left(\lambda^N_iN^\alpha t + \sum_{j\in[N]\backslash \{i\}}p^N_{ji} \mu_j^N N^\alpha\bar{B}_j^N(t) + \sum_{j\in\{0\}\cup[N]\backslash \{i\}}p^N_{ij} \mu_i^N N^\alpha\bar{B}_i^N(t)\right) \notag \\
&= \frac{1}{N^{\alpha}}\left(\lambda^N_i t + \sum_{j\in[N]\backslash \{i\}}p^N_{ji} \mu_j^N \bar{B}_j^N(t) + (1-p^N_{ii}) \mu_i^N \bar{B}_i^N(t)\right) \notag \\
&\leq \frac{t}{N^{\alpha}}\left(\lambda^N_i + \sum_{j=1}^N p^N_{ji} \mu_j^N + \mu_i^N \right), 
\label{eq:quadvarto0}
\end{align}
where we used that $\bar{B}_i^N(t) \le t$. Since $\bar{X}^N_i -\breve{X}^N_i$ is a martingale, applying Doob's maximal inequality, we obtain 
\begin{equation}
\label{eq:doobmax}
\E \sup_{0 \le t \le T} \left(\bar{X}^N_i(t) -\breve{X}^N_i(t)\right)^2 \le 4 \E \langle \bar{X}^N_i -\breve{X}^N_i \rangle_T.
\end{equation}
Combining \eqref{eq:quadvarto0} and \eqref{eq:doobmax} with Jensen's inequality yields
\begin{align*}
\E \sup_{0 \le t \le T} \left|\bar{X}^N_i(t) -\breve{X}^N_i(t)\right| &\le \left(\E \sup_{0 \le t \le T} \left(\bar{X}^N_i(t) -\breve{X}^N_i(t)\right)^2\right)^\frac{1}{2} \\
&\le 2 \left(\E \langle \bar{X}^N_i -\breve{X}^N_i \rangle_T\right)^\frac{1}{2} \\
& \le 2T^\frac{1}{2}N^{-\frac{\alpha}{2}}\left(\lambda^N_i + \sum_{j=1}^N p^N_{ji} \mu_j^N + \mu_i^N \right)^{\frac{1}{2}}. 
\end{align*}
Hence, we obtain 
\begin{align*}
\frac{1}{N} \sum_{i=1}^N \E \sup_{0  \le t \le T} \left|\bar{X}^N_i(t) -\breve{X}^N_i(t)\right| 
& \le 2T^\frac{1}{2}N^{-\frac{\alpha}{2}}\frac{1}{N}\sum_{i=1}^N\left(\lambda^N_i + \sum_{j=1}^N p^N_{ji} \mu_j^N + \mu_i^N \right)^{\frac{1}{2}} \\
& \le 2T^\frac{1}{2}N^{-\frac{\alpha}{2}}\sqrt{\frac{1}{N}\sum_{i=1}^N\left(\lambda^N_i + \sum_{j=1}^N p^N_{ji} \mu_j^N + \mu_i^N \right)} \\
& \le  \frac{2T^\frac{1}{2}}{N^{\alpha/2}}\sqrt{\|\lambda^N\|_1 + 2\|\mu^N\|_1}\,,
\end{align*}
where we used Jensen's inequality in the second inequality above and the $\lambda^N, \mu^N$ in the last line are to be understood as their piecewise functions on $[0,1]$. 
This gives the first claim.
Finally, applying Lemma \ref{lem:compqueue}, we obtain the second claim. 
\end{proof}

\subsection{Lifting the Intermediate Process to Infinite Dimension}

We now construct $\widetilde{Q}_u^N(t)$ from $\breve{Q}_i^N(t)$ as a piecewise function construction over the interval $[0,1]$: define 
\begin{equation}
\label{eq:infliftqueue}
\widetilde{Q}_u^N(t) := \sum_{i=1}^N \breve{Q}_i^N(t) {\bf1}_{u \in K^N_i }\,, \quad u \in [0,1], \,\, t \ge 0, 
\end{equation}
where $K^N_i$ is a partition of $[0,1]$ as in \eqref{eqn-GN-uv}. It follows that $\widetilde{Q}^N(0) \equiv \bar{Q}^N(0)$ as defined in \eqref{initQinft}. Similarly, define 
\begin{equation}\label{eq:infliftlm}
\widetilde{I}_u^N(t) := \sum_{i=1}^N \breve{I}_i^N(t) {\bf1}_{u \in K^N_i}\,, \quad u \in [0,1], \,\, t \ge 0. 
\end{equation}
It is evident that $\widetilde{Q}_u^N \in \mathcal{C}_T(L_1)$ and $ \widetilde{I}^N\in\mathcal{C}_T^\uparrow (L_1)$. From the construction, it's immediate that \begin{equation}
\label{eq:lift0}
\big|\breve{Q}_i^N(t) - \widetilde{Q}^N_{i/N}(t)\big|=0\,. 
\end{equation}

\begin{lemma}
With definitions \eqref{eq:infliftqueue}--\eqref{eq:infliftlm}, 
the lifted system $\widetilde{Q}_u^N(t)$ can be written as: for $u \in [0,1]$ and $t \ge 0$, 
\begin{equation*}
\widetilde{Q}_u^N(t) = \bar{Q}_u^N(0) + \lambda_u^Nt + \Big(\int_0^1 G^N(v,u)\mu^N_v dv - \mu^N_u \Big) t -  \int_0^1 G^N(v,u)  \mu^N_v \widetilde{I}^N_v  dv  +  \mu^N_u  \widetilde{I}_u^N(t).
\end{equation*}
Equivalently, we write
\begin{equation} \
\label{eq:lift1}
\widetilde{Q}^N (t) = \widetilde{X}^N (t) + (\mathbf{1}-(\boldsymbol{G}^N)^T) \widetilde{Y}^N (t), 
\end{equation}
where
\begin{align}
\label{eq:lift2}
\widetilde{X}^N(t)  & = \bar{Q}^N(0) + \lambda^N t - (\mathbf{1}-(\boldsymbol{G}^N)^T)\mu^Nt, \\ \widetilde{Y}^N (t) &= {\operatorname{diag}}(\mu^N)\widetilde{I}^N(t).\label{eq:lift3}
\end{align}
\end{lemma}

\begin{proof}
By \eqref{eqn-breveQNi} and  \eqref{eq:infliftqueue}, we obtain
\begin{align*} 
\widetilde{Q}_u^N(t)
&= \bar{Q}_u^N(0) + \lambda_u^Nt + \sum_{i=1}^N\left(\frac{1}{N}\sum_{j=1}^N   G^N_{ji} \mu_j^N - \mu_i^N\right) {\bf1}_{u \in K^N_i} t\\
& \quad -  \sum_{i=1}^N {\bf1}_{u \in K^N_i} \frac{1}{N}\sum_{j=1}^N   G^N_{ji} \mu_j^N \breve{I}_j^N(t) +\sum_{i=1}^N {\bf1}_{u \in K^N_i} {\mu_i^N \breve{I}_i^N(t)}\\
&= \bar{Q}_u^N(0) + \lambda_u^Nt + \left(\frac{1}{N}\sum_{j=1}^N\sum_{i=1}^N   G^N_{ji} \mu_j^N {\bf1}_{u \in K^N_i} - \sum_{i=1}^N\mu_i^N {\bf1}_{u \in K^N_i}\right)   t\\
&  \quad -  \frac{1}{N}\sum_{j=1}^N\sum_{i=1}^N     G^N_{ji} \mu_j^N \breve{I}_j^N(t) {\bf1}_{u \in K^N_i}+\sum_{i=1}^N  {\mu_i^N \breve{I}_i^N(t)}{\bf1}_{u \in K^N_i}\\
&= \bar{Q}_u^N(0) + \lambda_u^Nt + \Big(\int_0^1 G^N(v,u)\mu^N_v dv - \mu^N_u \Big) t -  \int_0^1 G^N(v,u)  \mu^N_v \widetilde{I}^N_v  dv  +  \mu^N_u  \widetilde{I}_u^N(t),
\end{align*}
where the final equality comes from
\begin{align*}
&\frac{1}{N}\sum_{j=1}^N\sum_{i=1}^N     G^N_{ji} \mu_j^N \breve{I}_j^N(t) {\bf1}_{u \in K^N_i} \\&=\sum_{i=1}^N  {\bf1}_{u \in K^N_i}
     \sum_{j=1}^N   G^N_{ji} \mu_j^N \breve{I}_j^N(t)\int_0^1 {\bf1}_{v\in K^N_j}  dv\\
&= \sum_{i=1}^N  {\bf1}_{u \in K^N_i}\int_0^1 
     \sum_{j=1}^N   G^N_{ji} \mu_j^N \breve{I}_j^N(t){\bf1}_{v\in K^N_j}  dv\\
&= \int_0^1  \bigg(\sum_{i=1}^N \sum_{j=1}^N G^N_{ji} {\bf1}_{u \in K^N_i} {\bf1}_{v\in K^N_j}\bigg)  \bigg(\sum_{j=1}^N \mu^N_j {\bf1}_{v\in K^N_j}\bigg)  \bigg(\sum_{j=1}^N \breve{I}^N_j(t){\bf1}_{v\in K^N_j}\bigg)  dv  \\
& = \int_0^1 G^N(v,u)  \mu^N_v \widetilde{I}^N_v  dv,
\end{align*}
and the identity $\frac{1}{N}\sum_{j=1}^N\sum_{i=1}^N   G^N_{ji} \mu_j^N {\bf1}_{u \in K^N_i} = \int_0^1 G^N(v,u)\mu^N_v dv$ follows similarly. 
\end{proof}

\begin{lemma}
Under Assumption \ref{assumption:GNGop}, $\widetilde{Q}^N$ and $\widetilde{Y}^N$ are the reflected process and regulator of $\widetilde{X}^N$ respectively. 
\end{lemma}

\begin{proof}
$(\widetilde{Q}^N,\widetilde{Y}^N)$ clearly satisfies \eqref{eq:cond1op}--\eqref{eq:cond3op} from the construction in  \eqref{eq:lift1}--\eqref{eq:lift2}. Hence by uniqueness of the Skorokhod problem, the claim holds. 
\end{proof}

\begin{lemma} \label{lem:lift-lip}
Under Assumption \ref{assumption:GNGop}, there exist some constants $\kappa_1, \kappa_2>0$ such that for sufficiently large $N$, 
    \begin{equation} \label{eq:lift-lip1}
\|\widetilde{Q}^N-\bar{Q}\|_{T,1} \leq \kappa_1\|\widetilde{X}^N-\bar{X}\|_{T,1} + \kappa_2\|(\boldsymbol{G}^N)^T-\boldsymbol{G}^T\|_{op}\,. 
    \end{equation} 
\end{lemma}

\begin{proof}
    The claim follows from Theorem \ref{thm:kconvergenceseq} applied to $(\boldsymbol{G}^N)^T$ and $\boldsymbol{G}^T$. 
\end{proof}

\begin{lemma} \label{lem:tQbQ-diff}
    Under Assumptions \ref{assumption:GNGop} and  \ref{assumption:Nbound}, we have for some $C''_T>0$, 
    \[
    \|\widetilde{Q}^N-\bar{Q}\|_{T,1} \leq C''_T N^{-\alpha/2}. 
    \]
\end{lemma}
\begin{proof}
First, by \eqref{eq:lift2} and \eqref{eq:Xbar-new}, 
    \begin{align*}
&\|\widetilde{X}^N - \bar{X}\|_{T,1} 
\\&\leq \|\bar{Q}^N(0)-\bar{Q}(0)\|_1 + T\|\lambda^N-\lambda\|_1 + T\|\mathbf{1}-(\boldsymbol{G}^N)^T\|_{op}\|\mu^N-\mu\|_1 + T\|(\boldsymbol{G}^N)^T-\boldsymbol{G}^T\|_{op}\|\mu\|_1 \\
&\leq \|\bar{Q}^N(0)-\bar{Q}(0)\|_1 + T\|\lambda^N-\lambda\|_1 + 2T\|\mu^N-\mu\|_1 + T\|(\boldsymbol{G}^N)^T-\boldsymbol{G}^T\|_{op}\|\mu\|_1. 
\end{align*}

Then, by \eqref{eq:lift-lip1}, we have 
\begin{align*}
\|\widetilde{Q}^N-\bar{Q}\|_{T,1} &\leq \kappa_1\|\widetilde{X}^N-\bar{X}\|_{T,1} + \kappa_2\|(\boldsymbol{G}^N)^T-\boldsymbol{G}^T\|_{op} \\
&\leq \|\bar{Q}^N(0)-\bar{Q}(0)\|_1 + T\|\lambda^N-\lambda\|_1 + 2T\|\mu^N-\mu\|_1 + \kappa_3\|(\boldsymbol{G}^N)^T-\boldsymbol{G}^T\|_{op}
\end{align*}
for some $\kappa_3<\infty$. Finally, the claim follows from  Assumptions \ref{assumption:GNGop} and\ref{assumption:Nbound}. 
\end{proof}

\subsection{Completing the Proofs of Theorem \ref{thm-main} and Corollaries \ref{coro:bQNbQ-conv0} and \ref{coro:example}}
\label{thmprf:fullconv}

\begin{proof}[Proof of Theorem \ref{thm-main} ] 
First,  we bound the LHS of \eqref{eqn:main1} by utilizing the intermediate process as follows:
\begin{align*}
&\frac{1}{N} \sum_{i=1}^N \E \sup_{0 \le t \le T} \bigg|\bar{Q}_i^N(t)-N\int_{(i-1)/N}^{i/N}\bar{Q}_u(t)\,du \bigg| \notag\\
&\le \frac{1}{N} \sum_{i=1}^N \E \sup_{0 \le t \le T} \big|\bar{Q}_i^N(t)-\breve{Q}_i^N(t)\big| \\
& \quad + \frac{1}{N} \sum_{i=1}^N  \sup_{0 \le t \le T} \big|\breve{Q}_i^N(t)-\widetilde{Q}_{i/N}^N(t)\big| \\
& \quad + \frac{1}{N} \sum_{i=1}^N \sup_{0 \le t \le T} \bigg|\widetilde{Q}_{i/N}^N(t)-N\int_{(i-1)/N}^{i/N}\bar{Q}_u(t)\,du\bigg|\,. 
\end{align*}
For the first term in the upper bound, by Lemmas \ref{lem:compqueue} and \ref{lem:compbound}, we obtain
\begin{equation*}
\frac{1}{N} \sum_{i=1}^N \E \sup_{0 \le t \le T} \left|\bar{Q}^N_i(t) -\breve{Q}^N_i(t)\right| 
\le \frac{CT^\frac{1}{2}}{N^{\alpha/2}}\sqrt{\|\lambda^N\|_1 + 2\|\mu^N\|_1} 
\le C'_T N^{-\alpha/2}
\end{equation*}
for some $C'_T>0$, 
where the second inequality follows from Assumption \ref{assumption:Nbound}. 

The second term is equal to zero by \eqref{eq:lift0}. 
For the third term, we have
\begin{align*}
     \frac{1}{N} \sum_{i=1}^N \sup_{0 \le t \le T} \bigg|\widetilde{Q}_{i/N}^N(t)-N\int_{(i-1)/N}^{i/N}\bar{Q}_u(t)\,du \bigg| & = \frac{1}{N} \sum_{i=1}^N \sup_{0 \le t \le T} \bigg|N\int_{(i-1)/N}^{i/N}\widetilde{Q}_{u}^N(t)-\bar{Q}_u(t)\,du\bigg|\\
     &\le \sum_{i=1}^N  \int_{(i-1)/N}^{i/N}\sup_{0 \le t \le T}\big|\widetilde{Q}_{u}^N(t)-\bar{Q}_u(t)\big|\,du \\
     &= \|\widetilde{Q}^N-\bar{Q}\|_{T,1} \le C'' N^{-\alpha/2}
\end{align*}
for some constant $C''>0$, 
where the last inequality follows from Lemma \ref{lem:tQbQ-diff}. 
Finally, the bound in \eqref{eqn:main1} follows from combining the properties above.  This completes the proof. 
\end{proof}

\begin{lemma}
    \label{lem:Qbar-bound}
    Suppose $\|\boldsymbol{G}\|_{op}<\infty$ and that $\lambda$ and $\mu$ are bounded.
    Then
    $$\sup_{0 \le u \le 1} \sup_{0 \le t \le T} \bar{Q}_u(t) < \infty.$$
\end{lemma}

\begin{proof}
    Fix $u \in [0,1]$ and $t \in [0,T]$.
    From \eqref{eq:Qbar}, \eqref{eq:Xbar}, and the explicit form of the solution of one-dimensional Skorokhod problem, we have
    \begin{align*}
        \mu_u \bar{I}_u(t) & = \max\left\{0, \sup_{0 \le s \le t} \left(-\bar{X}_u(s) + \int_0^1 \mu_v \bar{I}_v(s) G(v,u) \,dv\right)\right\} \\
        & \le \max\left\{0, \sup_{0 \le s \le t} \left(\mu_u s + \int_0^1 \mu_v \bar{I}_v(s) G(v,u) \,dv\right)\right\} \\
        & \le \mu_u t + \int_0^1 \mu_v \bar{I}_v(t) G(v,u) \,dv,
    \end{align*}
    where the second line uses the observation that $\bar{X}_u(s) \ge -\mu_u s$.
    Combining this with \eqref{eq:Qbar} and \eqref{eq:Xbar} gives
    \begin{equation*}
        \bar{Q}_u(t) \le \bar{X}_u(t) + \mu_u t \le \bar{Q}_u(0) + \left(\lambda_u + \int_0^1 \mu_v  G(v,u) \,dv\right) t. 
    \end{equation*}
    Since $\lambda$ and $\mu$ are bounded and $\|\boldsymbol{G}\|_{op}<\infty$, we have the desired result.    
\end{proof}

Now we are ready to prove Corollary \ref{coro:bQNbQ-conv0}.

\begin{proof}[Proof of Corollary \ref{coro:bQNbQ-conv0}]
    Using \eqref{eqn:main1}, it suffices to show
    $$\lim_{N \to \infty} \frac{1}{N} \sum_{i=1}^N \sup_{0 \le t \le T} \left|N\int_{(i-1)/N}^{i/N} \bar{Q}_u(t)\,du - \bar{Q}_{i/N}(t)\right| = 0.$$
    
    For this, note that
    \begin{align*}
        \frac{1}{N} \sum_{i=1}^N \sup_{0 \le t \le T} \left|N\int_{(i-1)/N}^{i/N} \bar{Q}_u(t)\,du - \bar{Q}_{i/N}(t)\right|
        & = \sum_{i=1}^N \sup_{0 \le t \le T} \left| \int_{(i-1)/N}^{i/N} \left( \bar{Q}_u(t) - \bar{Q}_{\lceil Nu \rceil/N}(t) \right) du \right| \\
        & \le \int_0^1 \sup_{0 \le t \le T} \left| \bar{Q}_u(t) - \bar{Q}_{\lceil Nu \rceil/N}(t) \right| du.
    \end{align*}
    Since the mapping $[0,1] \ni u \mapsto \{\bar{Q}_u(t) : t \in [0,T]\} \in \mathcal{C}_T$ is assumed to be continuous a.e., the above integrand goes to $0$ as $N \to \infty$ for a.e.\ $u \in [0,1]$.
    By the dominated convergence theorem and Lemma \ref{lem:Qbar-bound}, the above integral converges to $0$ as $N \to \infty$.
    This completes the proof. \qedhere
\end{proof}

\begin{proof}[Proof of Corollary \ref{coro:example}]
    In view of Corollary \ref{coro:bQNbQ-conv0}, it suffices to show that the mapping $[0,1] \ni u \mapsto \{\bar{Q}_u(t) : t \in[0,T]\} \in \mathcal{C}_T$ is continuous a.e.
    From \eqref{eq:idletime}, \eqref{eq:Qbar} and \eqref{eq:Xbar} we see that $\bar{Q}_u(t)$ is the reflected process of the free process
    $$\bar{Y}_u(t) := \bar{Q}_u(0) + \lambda_u t -\left(\mu_u  - \int_0^1 \mu_v  G(v,u) \,dv \right)t-\int_0^1 \mu_v\bar{I}_v(t)G(v,u)\,dv$$
    for the one-dimensional Skorokhod problem reflected at $0$.
    Since $G(u,v)$ is assumed to be continuous in a.e.\ $(u,v)$, we have that for a.e.\ $u$, the map $(u,v) \mapsto G(v,u)$ is continuous for a.e.\ $v$.
    Since $\bar{Q}_u(0),\lambda_u,\mu_u$ are assumed to be continuous in a.e.\ $u$, $\mu$ is assumed to be bounded, and $\bar{I} \in \mathcal{C}_T^\uparrow(L_1)$, it then follows from the dominated convergence theorem that $[0,1] \ni u \mapsto \{\bar{Y}_u(t) : t \in [0,T]\} \in \mathcal{C}_T$ is continuous in a.e.\ $u$.
    This completes the proof.
\end{proof}

\subsection{Convergence of the Mean Empirical Measure}
\label{subsec:empconv}

\begin{proof}[Proof of Theorem \ref{thm:main2}]
Denote by Lip$_1$ the collection of functions $f \colon \mathcal{D}_T \to \mathbb{R}$ that are $1$-Lipschitz with respect to the uniform metric, that is, $|f(x)-f(y)| \le \sup_{0 \le t \le T} |x(t)-y(t)|$ for all $x,y \in \mathcal{D}_T$.
Define $\widetilde{W}_1$ by
\begin{equation*}
    \widetilde{W}_1(\mu,\nu) := \sup_{f \in \text{Lip}_1} \left(\int_{\mathcal{D}_T} f \,d\mu - \int_{\mathcal{D}_T} f \,d\nu \right), \quad \mu,\nu \in \mathcal{P}(\mathcal{D}_T).
\end{equation*}
It then suffices to show that
\begin{equation} \label{eq:nuN-conv-claim}
\widetilde{W}_1(\nu^N,\bar\nu) \to 0 \quad \mbox{in probability as } N \to \infty. 
\end{equation}
Recall the lifted intermediate process $\widetilde{Q}_{u}^N(t)$  in \eqref{eq:infliftqueue}. 
Let
$$\widetilde{\nu}^N = \frac{1}{N} \sum_{i=1}^N \delta_{\widetilde{Q}_{i/N}^N(\cdot)}.$$
Using the estimates between the finite system and the lifted intermediate system in Lemma \ref{lem:compbound} and \eqref{eq:lift0}, we have
\begin{align*}
    \E \, \widetilde{W}_1(\nu^N,\widetilde{\nu}^N) & = \E \sup_{f \in \text{Lip}_1} \left| \frac{1}{N} \sum_{i=1}^N [f(\bar{Q}^N_i)-f(\widetilde{Q}_{i/N}^N)] \right| \\
    &\le \frac{1}{N} \sum_{i=1}^N \E \sup_{0 \le t \le T} |\bar{Q}_i^N(t)-\widetilde{Q}^N_{i/N}(t)| \to 0
\end{align*} 
as $N \to \infty$.
Also using the estimate between the intermediate system and the limiting system in Lemma \ref{lem:lift-lip} we have
\[
    \widetilde{W}_1(\widetilde{\nu}^N,\nu) = \sup_{f \in \text{Lip}_1} \left| \frac{1}{N} \sum_{i=1}^N f(\widetilde{Q}^N_{i/N}) -\int_0^1 f(\bar{Q}_u)\,du \right| \le \|\widetilde{Q}^N-\bar{Q}\|_{T,1} \to 0
\]
as $N \to \infty$.
Combining these two gives \eqref{eq:nuN-conv-claim} and completes the proof.
\end{proof}

\medskip

\appendix
\section{Supporting Lemmas and Proofs}
\label{sec:appendix}

\subsection{Some Useful Lemmas on Operators} \label{app:useful}

\begin{lemma}
\label{lem:opnorm}
For $f \in \mathcal{D}_T(L_1)$ and integral operator $\boldsymbol{F}$, 
$$\|\boldsymbol{F}f\|_{T,1} \leq \|\boldsymbol{F}\|_{op}\|f\|_{T,1}. 
$$
\end{lemma}
\begin{proof}
Using the definitions of operator $\boldsymbol{F}$ and norm $\|\cdot\|_{T,1}$, we have
\begin{align*}
    & \|\boldsymbol{F}f\|_{T,1} = \int_0^1\sup_{t\in[0,T]}\left|(\boldsymbol{F}f)_u(t)\right|\,du
    = \int_0^1 \sup_{t\in[0,T]}\left| \int_0^1F(u,v)f_v(t)\,dv \,\right|du\\
    &\leq \int_0^1 \sup_{t\in[0,T]} \int_0^1 \left|F(u,v)f_v(t)\right|\,dv \,du
    \leq \int_0^1 \int_0^1 \sup_{t\in[0,T]} \left|F(u,v)f_v(t)\right|\,dv \,du \\
    &\leq \int_0^1 \int_0^1  \left|F(u,v)\right|du \sup_{t\in[0,T]}\left|f_v(t)\right| dv\\
    & \leq \|\boldsymbol{F}\|_{op}\int_0^1\sup_{t\in[0,T]}\left|f_v(t)\right|dv = \|\boldsymbol{F}\|_{op}\|f\|_{T,1} \,. 
\end{align*}
\end{proof}

\begin{lemma}
\label{lem:compOpNorm}
For any integral operators $\boldsymbol{F}_1,\boldsymbol{F}_2$, $\|\boldsymbol{F}_1  \boldsymbol{F}_2\|_{op}\leq\|\boldsymbol{F}_1\|_{op}\cdot\|\boldsymbol{F}_2\|_{op}$. 
\end{lemma}

\begin{proof}
Take $f\in L_1([0,1])$. We have 
\begin{align*}
(\boldsymbol{F}_1 \boldsymbol{F}_2f)(u) 
&=\int_0^1F_1(u,w)\int_0^1 F_2(w,v)f(v)\,dv\,dw\\
&=\int_0^1\left(\int_0^1F_1(u,w) F_2(w,v)\,dw\right)f(v)\,dv\,. 
\end{align*}
Therefore,
\begin{align*}
\|\boldsymbol{F}_1  \boldsymbol{F}_2\|_{op} &= \sup_{v\in[0,1]}\int_0^1\bigg|\int_0^1F_1(u,w) F_2(w,v)dw\bigg|du\\
&\leq \sup_{v\in[0,1]}\int_0^1\int_0^1|F_1(u,w)| \cdot |F_2(w,v)|dwdu\\
& = \sup_{v\in[0,1]}\int_0^1\left(\int_0^1|F_1(u,w)|du\right) |F_2(w,v)|dw\\
&\leq \sup_{v\in[0,1]}\int_0^1\left(\sup_{z\in[0,1]}\int_0^1|F_1(u,z)|du\right) |F_2(w,v)|dw\\
&=\|\boldsymbol{F}_1\|_{op}\sup_{v\in[0,1]}\int_0^1|F_2(w,v)|dw\\
&=\|\boldsymbol{F}_1\|_{op}\cdot\|\boldsymbol{F}_2\|_{op}\,. 
\end{align*}
\end{proof}

\subsection{Supporting Proofs for the Infinite-Dimensional Skorokhod Problem}
\label{sec:fixed-point-existence-pf}

In this section, we prove Theorem \ref{thm:monotone-fixed-point}.
The argument is analogous to that of Theorem 14.2.2 of \cite{Whitt2002}, by using the fixed-point result on order sets established in Theorem 3.8.1 of \cite{edwards2012functional}.
However, adjustments have to be made for the infinite dimensional setup here.

We will follow the notation in Section 0.1.4 of \cite{edwards2012functional} and apply to the space $\mathcal{D}_T^\uparrow(L_1)$.
Consider the component-wise partial order on $L_1$ and $\mathcal{D}_T^\uparrow(L_1)$, that is, $f \equiv (f_u)_{u \in [0,1]} \le g \equiv (g_u)_{u \in [0,1]}$ in $L_1$ if $f_u \le g_u$ in $\mathbb{R}$ for all $u \in [0,1]$, and $x \equiv (x_u(t)) \le y \equiv (y_u(t))$ in $\mathcal{D}_T^\uparrow(L_1)$ if $(x_u(t))_{u \in [0,1]} \le (y_u(t))_{u \in [0,1]}$ in $L_1$ for all $t \in [0,T]$.
Let $\{x^n\}_{n \in \mathbb{N}}$ be a sequence in $\mathcal{D}_T^\uparrow(L_1)$. 
An element $x \in \mathcal{D}_T^\uparrow(L_1)$ majorizes, or is a majorant of, $\{x_n\}$ if and only if $x_n \le x$ for each $n \in \mathbb{N}$.
An element $x_0 \in \mathcal{D}_T^\uparrow(L_1)$ is the supremum of $\{x_n\}$ in $\mathcal{D}_T^\uparrow(L_1)$ if and only if $x_0$ is a majorant of $\{x_n\}$ and $x_0 \le x$ for any majorant $x$.

\begin{lemma} \label{lem:monotone-fixed-point1}
    Suppose $x^n \in \mathcal{D}_T^\uparrow(L_1)$ is majorized by some $y \in \mathcal{D}_T^\uparrow(L_1)$ and $\{x^n\}$ is increasing.
    The supremum of $\{x^n\}$ exists. 
\end{lemma}

\begin{proof}
Let $\tilde{x}_u(t) := \lim_{n \to \infty} x^n_u(t)$ for each $t \in [0,T]$ and $u \in [0,1]$, which exists and is finite because $\{x^n_u(t)\}$ is increasing and bounded from above by $y_u(t)$.
    Then $t \mapsto \tilde{x}_u(t)$ is increasing so that both one-sided limits exist.
    Let $x_u(t):=\tilde{x}_u(t+)$ for $0 \le t < T$ and $x_u(T):=\tilde{x}_u(T)$.
    It is not difficult to show that $x_u(t+)=x_u(t)$ for $0 \le t < T$.
    As for $t=T$, since $x^n_u(T)=x^n_u(T-)$, we must have $\tilde{x}_u(T)=\tilde{x}_u(T-)$ and hence $x_u(T)=x_u(T-)$.
    Therefore $x \in \mathcal{D}_T^\uparrow(L_1)$. 
    Clearly, $x$ is the supremum of $\{x^n\}$.
\end{proof}

The following are some properties of $\pi \equiv \pi_{X,F}$ introduced in Definition \ref{def:pidefD} for a given $X\in \mathcal{D}_T(L_1)$ and $\boldsymbol{F}\in\mathcal{R}$.

\begin{lemma}\label{lem:monotone-fixed-point2}
    \begin{itemize} 
    \item []
\item  [(i)]
If $x \in \mathcal{D}_T^\uparrow(L_1)$, then $\pi(x) \in \mathcal{D}_T^\uparrow(L_1)$.

\item [(ii)]
 If $x \le y$ in $\mathcal{D}_T^\uparrow(L_1)$, then $\pi(x) \le \pi(y)$ in $\mathcal{D}_T^\uparrow(L_1)$.
    
\item [(iii)]
If $x^n \uparrow x$ in $\mathcal{D}_T^\uparrow(L_1)$, then $\pi(x^n) \uparrow \pi(x)$ in $\mathcal{D}_T^\uparrow(L_1)$. 
\end{itemize}
\end{lemma}

\begin{proof}
    (i) By the dominated convergence theorem, $(\boldsymbol{F}x)_u(t+)=(\boldsymbol{F}x)_u(t)$ and $(\boldsymbol{F}x)_u(T-)=(\boldsymbol{F}x)_u(T)$.
    Also, $\|\pi(x)\|_{T,1} \le \|X\|_{T,1} + \|\boldsymbol{F}\|_{op} \|x\|_{T,1} < \infty$.
    So $\pi(x) \in \mathcal{D}_T^\uparrow(L_1)$. 

    (ii) By part (i), $\pi(x), \pi(y) \in \mathcal{D}_T^\uparrow(L_1)$. Since $F \ge 0$, the definition of $\pi$ guarantees that $\pi(x) \le \pi(y)$.

    (iii)  Since $x \ge x^n$, we have $\pi(x) \ge \pi(x^n)$ for each $n$, by part (ii).
    So $\pi(x)$ is an upper bound of $\{\pi(x^n)\}$.
    Let $y$ be any upper bound of $\{\pi(x^n)\}$, namely $y \ge \pi(x^n)$ for each $n$.
    To show that $\pi(x^n) \uparrow \pi(x)$, it remains to show that $y \ge \pi(x)$.

    Now fix $u \in [0,1]$, $t \in [0,T)$, $\varepsilon \in (0,1)$, and $K \in \mathbb{N}$. The case $t=T$ can be argued in a similar manner.
    Recall the proof of Lemma \ref{lem:monotone-fixed-point1}.  
    Letting $\tilde{x}_u(t) = \lim_{n \to \infty} x^n_u(t)$, we have $x_u(t) = \tilde{x}_u(t+)$.
    By the definition of $\pi$, there exists some $s \in [0,t]$ such that
    $$\pi(x)_u(t) \le [-X_u(s) + (\boldsymbol{F}x)_u(s)]^+ + \varepsilon.$$
    By the dominated convergence theorem, we have
    $$(\boldsymbol{F}x)_u(s) = (\boldsymbol{F}\tilde{x})_u(s+).$$
    By the right-continuity, there exists some $\delta \in [0,\frac{T-t}{K})$ such that
    $$(\boldsymbol{F}x)_u(s) \le (\boldsymbol{F}\tilde{x})_u(s+\delta)+\varepsilon, \quad -X_u(s) \le -X_u(s+\delta)+\varepsilon.$$
    By the definition of $\tilde{x}$ and the dominated convergence theorem, there exists some $n$ such that
    $$(\boldsymbol{F}\tilde{x})_u(s+\delta) \le (\boldsymbol{F}x^n)_u(s+\delta) + \varepsilon.$$
    Combining above estimates gives
    \begin{align*}
        \pi(x)_u(t) & \le [-X_u(s+\delta) + (\boldsymbol{F}x^n)_u(s+\delta)]^+ + 4\varepsilon \\
        & \le \pi(x^n)_u\bigg(t+\frac{T-t}{K}\bigg) + 4\varepsilon \le y_u\bigg(t+\frac{T-t}{K}\bigg)+4\varepsilon.
    \end{align*}
    Letting $\varepsilon \to 0$ and $K \to \infty$ gives
    $$\pi(x)_u(t) \le y_u(t).$$
        This shows that $\pi(x)$ is the least upper bound of $\{\pi(x^n)\}$ and hence completes the proof.
\end{proof}

\begin{proof}[Proof of Theorem \ref{thm:monotone-fixed-point}]
    From Lemma \ref{lem:monotone-fixed-point2}(ii), we have $\pi(x^0) \le \pi(y^0)$.
    Using this and the assumptions that $x^0 \le \pi(x^0)$ and $\pi(y^0) \le y^0$ we have by induction that
    $$x^0 \le x^1 \le \dotsb \le x^n \le y^0.$$
    Thus $\{x^n\}$ is increasing and majorized by $y^0$, and therefore $x^n \uparrow x$ for some $x \in \mathcal{D}_T^\uparrow(L_1)$ by Lemma \ref{lem:monotone-fixed-point1}.
    From Lemma \ref{lem:monotone-fixed-point2}(iii), we have $\pi(x^n) \uparrow \pi(x)$.
    On the other hand, $\pi(x^n) = x^{n+1} \uparrow x$.
    Therefore $x=\pi(x)$.
\end{proof}

\subsection{Additional Proofs}
\label{sec:additional-pf}

\begin{proof}[Proof of Lemma \ref{lem:GNinR}]
First, note that
\begin{align*}
\|(\boldsymbol{G}^N)^T\|_{op} &= \sup_{v \in [0,1]}\int_0^1  |G^N(v,u)|du 
\\&= \sup_{v \in [0,1]}\int_0^1  \sum_{i=1}^N \sum_{j=1}^N G^N_{ji} {\bf1}_{u \in K^N_i} {\bf1}_{v\in K^N_j}du \\
& = \sup_{v \in [0,1]}  \sum_{i=1}^N \sum_{j=1}^N G^N_{ji} {\bf1}_{v\in K^N_j}\int_0^1{\bf1}_{u \in K^N_i}du \\
&= \max_{j\in[N]}  \sum_{i=1}^N \frac{1}{N} G^N_{ji} \leq 1\,.
\end{align*}
For the spectral radius, let $f$ be an eigenfunction of  $(\boldsymbol{G}^N)^T$ with corresponding eigenvector $\lambda$. Then
\begin{align*}
\lambda f(u) &= \int_0^1G^N(v,u)f(v)dv \\
&= \int_0^1\sum_{i=1}^N \sum_{j=1}^N G^N_{ji} {\bf1}_{u \in K^N_i} {\bf1}_{v\in K^N_j}f(v)dv \\
&= \sum_{i=1}^N \sum_{j=1}^N G^N_{ji} {\bf1}_{u \in K^N_i}\int_0^1   {\bf1}_{v\in K^N_j}f(v)dv,
\end{align*}
and hence, for any $l=1,\dots,N$,
\begin{align*}
\lambda \int_0^1{\bf1}_{u \in K^N_l}f(u)du &= \int_0^1\sum_{i=1}^N \sum_{j=1}^N G^N_{ji} {\bf1}_{u \in K^N_i}{\bf1}_{u \in K^N_l}\int_0^1   {\bf1}_{v\in K^N_j}f(v)dv\,du \\
&= \sum_{i=1}^N \sum_{j=1}^N G^N_{ji} \int_0^1{\bf1}_{u \in K^N_i}{\bf1}_{u \in K^N_l}du\int_0^1   {\bf1}_{v\in K^N_j}f(v)dv \\
&= \sum_{j=1}^N \frac{1}{N} G^N_{jl} \int_0^1   {\bf1}_{v\in K^N_j}f(v)dv\,. 
\end{align*}
Writing $\mathbf{c}=(c_1,c_2,\dots,c_N)$, where $c_i = \int_0^1{\bf1}_{u \in K^N_i}f(u)du$, we have 
\begin{align*}
&\lambda c_i = \sum_{j=1}^N \frac{1}{N} G^N_{jl}c_j,
\end{align*}
or in vector form, $\lambda \mathbf{c} = \frac{1}{N}(G^N)^T\mathbf{c}\,. $
Therefore, if $\lambda$ is an eigenvalue of $(\boldsymbol{G}^N)^T$,  then $N\lambda$ is an eigenvalue of $(G^N)^T$. Under Assumption \ref{assumption:GNGop}, as the spectral radius of the finite matrix $G^N$ (and thus also $(G^N)^T$) satisfies $\rho(G^N)<N$, we have $\rho((\boldsymbol{G}^N)^T) \leq \frac{1}{N}\rho((G^N)^T) < 1$ (in fact, it holds that $\rho((\boldsymbol{G}^N)^T) = \frac{1}{N}\rho((G^N)^T)$ holds but we do not need it here). Therefore, $(\boldsymbol{G}^N)^T \in \mathcal{R}$.
\end{proof}

\begin{proof} [Proof of Lemma \ref{lem:compqueue}, continued] 
In order to prove the claim we must show that matrix multiplication is equivalent to operator composition of their corresponding blockwise constant operators. Consider block operator $\boldsymbol{A}^N$ created from matrix $A=\{A_{ij}\}_{i,j=1,\cdots N}$. For any $f\in\mathcal{D}_T(L_1)$, write $f_u(t) = \sum_{i=1}^Nf_u(t)\mathbf{1}_{u\in K_i^N}$. 
For $u\in K_i^N$, $\boldsymbol{A}^N$ acts on $f$ as 
$$(\boldsymbol{A}^Nf)_u(t) =\sum_{j=1}^NA_{ij}\int_{K_j^N}f_v(t)\,dv.$$
Let $\boldsymbol{B}^N$ be a similarly defined block operator with corresponding matrix $B$. For $u\in K_i^N$, the composition of $\boldsymbol{B}^N$ with $\boldsymbol{A}^N$ acts on $f$ as 
\begin{align*}
(\boldsymbol{B}^N\boldsymbol{A}^Nf)_u(t) &= \sum_{k=1}^N\sum_{j=1}^NB_{ik}A_{kj}\frac{1}{N}\int_{K_j^N}f_v(t)\,dv \\
&= \sum_{j=1}^N(BA)_{ij}\frac{1}{N}\int_{K_j^N}f_v(t)\,dv.
\end{align*}

Writing $\hat{A}=\frac{1}{N}A$ and $\hat{B}=\frac{1}{N}B$ yields for $u\in K_i^N$,
\begin{align*}
(\boldsymbol{A}^Nf)_u(t) &=N\sum_{j=1}^N\hat{A}_{ij}\int_{K_j^N}f_v(t)\,dv, \\
(\boldsymbol{B}^N\boldsymbol{A}^Nf)_u(t) &= N\sum_{j=1}^N(\hat{B}\hat{A})_{ij}\int_{K_j^N}f_v(t)\,dv.
\end{align*}
That is, with the scaling above, composition of piecewise constant operators acts as matrix multiplication of their counterparts. From this, it follows that $(G^N)^{(k)}(u,v)=\frac{1}{N^k}(G^N)^k_{ij}\mathbf{1}_{u\in K_i^N}\mathbf{1}_{v\in K_j^N}$. 
Hence, it follows that $$\|((\boldsymbol{G}^N)^T)^{(k)}\|_{op} = \max_{1\le i\le N}\sum_{j=1}^N(P^N)^k_{ij}\,,$$
thus completing the claim.
\end{proof}

\bibliographystyle{abbrv}
\bibliography{reference}

\end{document}